\documentclass[11pt]{article}

\def\C{{\mathcal C}}

\def\F{{\mathcal F}}

\def\P{{\mathcal P}}

\def\R{{\mathbb R}}

\def\al{\alpha}
\def\de{\delta}

\def\la{\lambda}

\def\OM{\Omega}

\def\Gh{\hat{G}}

\def\ap{\rightarrow}

\def\es{\emptyset}

\def\bz{{\bf 0}}

\def\fa{\; \forall}

\def\as{\mbox{ a.s.}}

\def\nm{\Vert}

\renewcommand{\and}{\mbox{$\wedge$}}

\newcommand{\bc}{\begin{center}}
\newcommand{\ec}{\end{center}}
\newcommand{\be}{\begin{equation}}
\newcommand{\ee}{\end{equation}}
\newcommand{\bd}{\begin{displaymath}}
\newcommand{\ed}{\end{displaymath}}
\newcommand{\ba}{\begin{array}}
\newcommand{\ea}{\end{array}}
\newcommand{\ben}{\begin{enumerate}}
\newcommand{\een}{\end{enumerate}}
\newcommand{\bit}{\begin{itemize}}
\newcommand{\eit}{\end{itemize}}
\newcommand{\beq}{\begin{eqnarray}}
\newcommand{\eeq}{\end{eqnarray}}
\newcommand{\btab}{\begin{tabular}}
\newcommand{\etab}{\end{tabular}}
\newcommand{\bfig}{\begin{figure}}
\newcommand{\efig}{\end{figure}}
\newcommand{\btp}{\begin{tikzpicture}}
\newcommand{\etp}{\end{tikzpicture}}

\newcommand{\nmm}[1]{ \nm #1 \nm }
\newcommand{\nmeu}[1]{ \nm #1 \nm_2 }

\newcommand{\nmi}[1]{ \nm #1 \nm_\infty}

\def\nmsl1{\nm_{{\rm SL1}}}

\usepackage{latexsym}
\usepackage{amsfonts}
\usepackage{amsmath}
\usepackage{amsthm}
\usepackage{geometry} 
\usepackage{xcolor} 
\usepackage{amssymb}
\usepackage{parskip}
\usepackage{graphicx}
\usepackage{enumerate}
\usepackage{hyperref}
\usepackage{xspace,colortbl}

\theoremstyle{plain}
\newtheorem{lemma}{{\bf Lemma}}

\newtheorem{example}[lemma]{{\bf Example}}
\newtheorem{remark}{{\bf Remark}}
\newtheorem{definition}[lemma]{{\bf Definition}}
\newtheorem{theorem}[lemma]{{\sc Theorem}}

\newcommand{\mycomment}[1]{}
 \usepackage{xspace,colortbl}
\definecolor{blu}{rgb}{.01,.01,1}
\definecolor{dblu}{rgb}{.01,.01,.85}
\definecolor{dred}{rgb}{.7,.01,.01}
\definecolor{red}{rgb}{1,0,0}
\definecolor{grn}{rgb}{.01,.5,.01}

\newcommand{\Ind}{\mbox{{ \sf 1}}}
\newcommand{\Mp}{M^\prime}
\newcommand{\Mpp}{M^{\prime \prime}}
\usepackage{geometry} 
\begin{document} 

\title{Revisiting Stochastic Approximation and \\
Stochastic Gradient Descent}


\author{Rajeeva L.\ Karandikar, B.\ V.\ Rao and M.\ Vidyasagar
\thanks{RLK and BVR are with the Chennai Mathematical Institute;
MV is with the Indian Institute of Technology, Hyderabad.
Emails: rlk@cmi.ac.in, bvrao@cmi.ac.in, m.vidyasagar@iith.ac.in}
}

\maketitle

{

\begin{abstract}

In this paper, we introduce a new approach to proving the convergence
of the Stochastic Approximation (SA)
and the Stochastic Gradient Descent (SGD) algorithms.
The new approach is based on a concept called GSLLN (Generalized Strong
Law of Large Numbers), which extends the traditional SLLN.
Using this concept, we provide sufficient conditions for convergence,
which effectively decouple the properties of the function whose zero
we are trying to find, from the properties of the measurement errors (noise
sequence). 
The new approach provides an alternative to the two widely used approaches,
namely the ODE approach and the martingale approach, and also permits
a wider class of noise signals than either of the two known approaches.
In particular, the ``noise'' or measurement error \textit{need not}
have a finite second moment, and under suitable conditions, not even
a finite mean.
By adapting this method of proof, we also derive sufficient conditions
for the convergence of zero-order SGD, wherein the stochastic gradient
is computed using $2d$ function evaluations, but
no gradient computations.
The sufficient conditions derived here are the weakest to date,
thus leading to a considerable expansion of the applicability of SA and SGD
theory.

\end{abstract}

\section{Introduction}\label{sec:Intro}

\subsection{Historical Review}\label{ssec:Review}

In this paper, we revisit the problems of Stochastic Approximation (SA),
and its close relative, Stochastic Gradient Descent (SGD).
Until now, there have been two broad approaches to addressing these problems,
namely: the ODE approach, and the martingale approach.
In this paper we introduce another approach, which might be called the
Generalized Strong Law of Large Numbers (GSLLN) approach.
By building on a classical proof of the Strong Law of Large Numbers (SLLN)
given by Kolmogorov (see for example \cite[Theorem 3.3]{Breiman92}),
we are able to establish the convergence of both SA and SGD in more
general situations than is the case at present.
Specifically, both the ODE approach and the martingale approach require
that the measurement error (``noise'') should have finite conditional
variance at each iteration.
We can dispense with that requirement, under suitable conditions.

In the remainder of this section, we formulate the two problems under study,
starting with SA.

Suppose $d$ is a fixed integer, and let $G : \R^d \ap \R^d$ be
a map such that the equation $G(x) = \bz$ has a unique solution $x^*$.
Specific hypotheses on the function $G(\cdot)$ will be added by and by.
In order to find $x^*$, suppose we can only access noise-corrupted
measurements of $G(\cdot)$.
In a seminal paper \cite{Robbins-Monro51}, Robbins and Monro studied
the scalar case ($d = 1$), and proposed the following
iterative procedure, which they called ``stochastic approximation.''
However, to avoid repetition, we state the procedure for the case of
arbitrary $d$.
Choose any $x_0 \in \R^d$.
At iteration $n$, suppose we measure $G(x_n) + W_{n+1}$, where
$W_{n+1}$ is the measurement error, or ``noise.''
Then we define
\be\label{I10}
X_{n+1}=X_n - \beta_n (G(X_n)+W_{n+1}) , \fa n \geq 0 .
\ee
In the above, the number $\beta_n \in (0,\infty)$ is called the``step size.''
Note that in the Machine Learning literature, $\beta_n$ is also called
the ``learning rate.''
In \cite{Robbins-Monro51}, the authors analyzed the moments of the iterations
$X_n$, under some very strong assumptions on $G(\cdot)$, which have
been relaxed by subsequent researchers.
They also assumed that the noise sequence $\{ W_n \}$ is i.i.d., has zero mean,
and had finite variance.
With these assumptions,
They showed that $X_n \ap x^*$ in probability, provided the two conditions
\be\label{I0}
  \sum_n \beta_n=\infty, \quad \sum_n \beta_n^2<\infty
\ee
are satisfied.
Indeed, the conditions \eqref{I0} are known as Robbins-Monro conditions.
Subsequent relaxations of the assumptions on both $G(\cdot)$ and on
$\{ W_n \}$ are discussed a little later.

Now we discuss the SGD method.
Let $d$ be a fixed integer, and let $F: \R^d \ap \R$ be a $\C^1$ function.
The objective is to find a stationary point $x^*$ of $F(\cdot)$, that is,
a solution to $G(x) = \bz$, where $G(x) := \nabla F(x)$.
If one has access to noise-corrupted measurements of $G(\cdot)$, then
this problem can also be addressed via stochastic approximation, as
described above.
However, if one can measure only $F(\cdot)$ corrupted by noise,
alternate approaches are needed, as described next.
These are often referred to as ``zeroth-order'' or ``gradient-free'' methods.

The first approach was proposed by \cite{Kief-Wolf-AOMS52}, almost
immediately after the publication of \cite{Robbins-Monro51}.
The authors study the scalar case $d=1$, and
assume that the equation $G(Y) = 0$ has a unique solution $x^*$.
If $Y_n$ is the current guess for the solution, then Kiefer and Wolfowitz 
construct as ``approximate gradient'' at $Y_n$ via
\be\label{I11}
\Gh(Y_n) = \frac{1}{2 c_n} [ (F(Y_n + c_n) + \Mp_{n+1} )
- (F(Y_n - c_n) + \Mpp_{n+1})] ,
\ee
where $\{ \Mp_n \}$ and $\{ \Mpp_n \}$ are i.i.d.\ processes that are
also independent of each other.
Moreover, they assumed that both processes have zero mean and finite variance.
In this definition, the number $c_n$ is called the \textbf{increment.}
By writing the above as
\bd
\Gh(Y_n) = \frac{F(Y_n + c_n) - F(Y_n - c_n)}{2 c_n} 
+ \frac{\Mp_{n+1} - \Mpp_{n+1}}{2 c_n} ,
\ed
it can be observed that the first term equals $G(Y_n) + o(c_n)$.
Hence, in order to be effective, the increment $c_n$ has to approach zero.
As a result, the \textit{second} term has variance that is \textit{unbounded}
as $c_n \ap 0$.
The iterative scheme proposed in \cite{Kief-Wolf-AOMS52} is
\be\label{I12}
Y_{n+1} = Y_n - \eta_n \Gh(Y_n) ,
\ee
where, as before, $\eta_n$ is the step size.
They showed that, under suitable conditions on $G(\cdot)$, $Y_n$ converges in
probability to $x^*$, provided
\be\label{I13}
c_n \ap 0 , \quad \sum_{n=0}^\infty (\eta_n / c_n)^2 < \infty ,
\quad \sum_{n=0}^\infty \eta_n c_n < \infty ,
\quad \sum_{n=0}^\infty \eta_n = \infty .
\ee
Subsequently, Blum \cite{Blum54b,Blum54} extended the approach to higher dimensions, and also proved that the iterations converge almost surely,
provided \eqref{I13} hold.
For this reason, the conditions \eqref{I13} are referred to as the
Kiefer-Wolfowitz-Blum (KWB) conditions.

Until now, the two most commonly used approaches to analyzing the
behavior of SA and SGD are the ODE approach and the martingale approach.
The core idea that underlies the ODE approach is that the stochastic process
generated by the iterations \eqref{I10} ``converges'' in a precise
sense to the trajecftories of the \textit{deterministic} equation
\be\label{I14}
\dot{X} = - G(X) .
\ee
Note that if $G(x^*) = \bz$, then $x^*$ is an equilibrium of \eqref{I14}.
Thus, if $x^*$ is a \textit{globally attractive} equilibrium of \eqref{I14},
then under additional assumptions, it can be shown that $X_n \ap x^*$
almost surely.
The ODE approach was pioneered in \cite{Kushner-JMAA77}, in which
it is shown only that $X_n \ap x^*$ in probability.
The approach is developed at book-length in \cite{Kushner-Clark78}.
The conclusion is strengthened to almost sure convergence by Ljung;
see \cite{Ljung-TAC77a,Ljung78}.
A very readable summary of these earlier papers is given in
\cite{Meti-Priou84}.
The approach is extended to infinite-dimensional Hilbert spaces in
\cite{Kushner-Shwartz85}.
There are several book-length treatments of the ODE approach, namely
\cite{Kushner-Clark78,BMP90,Kushner-Yin97,Kushner-Yin03,Kushner-Clark12}.
A survey of the papers up to the time of its publication can be found
in \cite{Lai03}.
The most recent treatments of the ODE approach can be found in 
\cite{Borkar08,Borkar22}.

In all of the above papers and books, a central assumption is that
the iterations of \eqref{I10} are bounded almost surely.
The convergence proofs themselves do not offer any way to verify this
assumption.
A seminal paper in this direction is \cite{Borkar-Meyn00}, in which the
boundedness of the iterations is a \textit{conclusion and not a hypothesis}.

The second popular approach is based on the fundamental fact that
a nonnegative supermartingale converges almost surely.
Gladyshev \cite{Gladyshev65} was the first to 
prove the convergence of SA using this approach.
Note that Gladyshev's paper in 1965 predates Kushner's paper in 1977.
However, Gladyshev's paper apparently attracted relatively little
notice at the time.
Gladyshev's paper was the first to obtain the boundedness of
the iterations as a \textit{conclusion and not as a hypothesis.}
Specifically, he demonstrated that the square summability of $\beta_n$
in \eqref{I10} suffices to establish the boundedness of the iterations.
If, in addition, the sum of $\beta_n$ diverges, then the convergence of $X_n$
follows.
Robbins and Siegmund \cite{Robb-Sieg71} extended the approach of Gladyshev.
They called their stochastic processes ``almost supermartingales,''
and this is now standard terminology.
One advantage of the Robbins-Siegmund approach over the Gladyshev approach
is that the former is readily applicable to the case where the step sizes
$\beta_n$ are random, which is not the case with Gladyshev's approach.
Random step sizes arise naturally in ``Block'' stochastic approximation,
wherein, at each iteration, \textit{some but not necessarily all}
components of $X_n$ are updated.
The book \cite{BMP90} contains a discussion of several algorithms for SA,
including the almost-supermartingale approach.

The ability to provide separate proofs of boundedness and convergence
is characteristic of the martingale approach.
On the other hand, the ODE approach can handle the case where \eqref{I14}
has multiple equilibria, which is difficult in the martingale approach.

Some other relevant references that deal with these approaches are
\cite{BMP90,Bertsekas19,Borkar98,Borkar08,Dvoretzky56,Jaakkola-et-al94,RLK-BVR24,MV-RLK-BASA-arxiv21,MV-RLK-SGD-JOTA24,Milz23,Schmett60,Tsi-Van-TAC97}. 


\subsection{Contributions of the Paper}\label{ssec:Contrib}

Now we describe the contributions of the paper.
In one sentence, we introduce a third approach, as an alternative to the
ODE and martingale approaches.
Let $\{ \F_n \}$ be any filtration to which the noise process
$\{ W_n \}$ is adapted, and suppose further that $G_0$ is
measurable with respect to $\F_0$.
To analyze SA or SGD, most authors impose assumptions on the
conditional expectation and conditional variance of $W_{n+1}$ with
respect to $\F_n$.
Implicitly this assumes that the noise $W_n$ is square-integrable
(has a finite second moment).
In what follows, we do away with this assumption.

Specifically, we propose the concept of a stochastic process $\{ W_n\}$
satisfying the GSLLN (Generalized Strong Law of Large Numbers)
with respect to a specific ``rate,'' which in the present setting
is the step size sequence $\{ \beta_n \}$ or $\{ \eta_n \}$ depending
on the algorithm used.
By assuming that the noise process $\{ W_n \}$ satisfies a GSLLN with
respect to the rate, and that the function $G(\cdot)$ satisfies a few
additional assumptions, we can prove that both SA and SGD converge.
This generalization is applicable to the case where the noise $W_n$
may not have a finite variance, and in some cases, not even a mean.

Here are some specific examples of convergence results that can be obtained
using the present approach.
\ben
\item The GSLLN property holds if $\beta_n = 1/n$ and $\{ W_n \}$
satisfies any one of the following three properties:
\bit
\item $\{ W_n \}$ is an i.i.d.\ process with zero mean, and each $W_n$ is
absolutely integrable, that is, $E[ \nmm{W_1} ] < \infty$.
\item $W_n$'s are pairwise independent (but not identically distributed),
with $\lim_n E(W_n)=0$,  and there exists an $\al \in (1,2]$ such that
\bd
\sup_n[ E(\nmm{W_n}^\alpha)]<\infty .
\ed
\item $\{ W_n\}$ is a martingale difference sequence with 
\bd
\sup_n [ E(\nmm{W_n}^\alpha)]<\infty
\ed
for some $\alpha>1$.
\eit
\item If the noise sequence $\{ W_n \}$ satisfies any of the above three
conditions, then the Stochastic Approximation algorithm 
converges almost surely for the step size sequence $\beta_n = 1/n$, while the
Stochastic Gradient Descent algorithm converges almost surely for
the choices
\bd
\eta_n=\frac{1}{n\log(1+n)}, c_n=\frac{1}{\log(1+n)} .
\ed
\item
In the special case where  $\{W_n\}$ are i.i.d.\  with a 
symmetric distribution (that is, $W_1$ and $-W_1$ have the same distribution),
GSLLN holds if there exists a $\de \in (0,1]$ such that
\bd
E \left[ \frac{|W_{1,i}|}{{\log(1+|W_{1,i}|)^\delta}} \right] <\infty, \fa i .
\ed
Hence the Stochastic Approximation converges with 
\bd
\beta_n=\frac{1}{n\log(1+n)^\delta},\;\;\forall n\ge 1 .
\ed
Similarly, if there exists a $\de < 1$ such that
\bd
E \left[ \frac{|W_{1,i}|}{{\log(1+|W_{1,i}|)^\delta}} \right] <\infty 
\ed
then the Stochastic Gradient Descent algorithm will converge with
\bd
\eta_n=\frac{1}{n\log(1+n)} , c_n=\frac{1}{\log(1+n)^{1-\delta}} .
\ed
\een


\section{Main Results}\label{sec:Main}

In this section, we state the main results of the paper.
The proofs are deferred to Section \ref{sec:Proofs}.
We begin by introducing the GSLLN (Generalized Strong Law of Large Numbers)
property for a stochastic process $\{ W_n \}$.
Then we state that, if the GSLLN property holds, then both SA and SGD
converge, under suitable additional conditions.
This establishes the centrality of the GSLLN property.
Finally, we state several sufficient conditions for the GSLLN property
to hold.
 
\subsection{Generalized Strong Law of Large Numbers: Definition
}\label{ssec:GSLLN}

In this section, we introduce the key notion of GSLLN
(Generalized Strong Law of Large Numbers),
and then provide several sufficient conditions for ensuring that a
given process $\{ W_n \}$ satisfies the GSLLN with respect to a given
rate sequence $\{ \beta_n \}$.
The proofs follow the steps involved in Kolmogorov's proof of 
the Strong law of large numbers in the case of i.i.d. $L^1$ random variables;
see \cite{Breiman92,Chung51,Goswami-Rao25}.

To set up the definition,
let $(\F_n)_{n\ge 0}$ be a filtration on a probability space $(\Omega,\F,\P)$,
and let $\{ \beta_n \}$ be a sequence of positive numbers that satisfies
\be\label{U8}
\lim_{n \ap  \infty}\beta_n=0,\;\;\;\sum_{n=0}^\infty\beta_n=\infty.
\ee

\begin{definition}\label{def:GSLLN}
Let $( \F_n )$ be a filtration on a probability space $( \OM,\F,P)$.
Let $\{\beta_n:n\ge 0\}$ satisfy \eqref{U8}.
An $\R^d$ valued $(\F_n)_{n\ge 0}$  adapted sequence of Random variables
$\{W_n:n\ge 1\}$ is said to satisfy the \textbf{Generalized Strong Law of
Large Numbers} (GSLLN) at the \textbf{rate} $\{\beta_n:n\ge 0\}$
with respect to the filtration $(\F_n)_{n\ge 0}$, if  for all
uniformly bounded  $(\F_n)_{n\ge 0}$  adapted random variables  $\{\zeta_n :
n\ge 0\}$ and for all $t\in (0,\infty)$
the sequence $\{S_n:n\ge 0\}$ defined by $S_0=0$ and 
\be\label{U7}
S_{n+1}=(1-t\beta_n )S_n+t\beta_n \zeta_n W_{n+1},\;\;n\ge 0
\ee
converges to 0 almost surely.
\end{definition}

If there is a fixed filtration in consideration, we will drop it from
the statement and say it as $\{W_n:n\ge 1\}$ satisfies the
GSLLN at the {\em rate} $\{\beta_n:n\ge 0\}$.
Clearly, the set of  $(\F_n)_{n\ge 0}$ adapted sequences of random
variables that satisfy the GSLLN at  the {\em rate} $\{\beta_n:n\ge 0\}$
w.r.t.\ the filtration $(\F_n)_{n\ge 0}$ is a linear space.

Given a sequence of random variables  $\{W_n:n\ge 1\}$ on $(\Omega,\F,\P)$, let $(\F^W)$ denote the filtration $\F^W_0=\{\Phi,\Omega\}$, and for $n\ge 1$, $\F^W_n=\sigma(W_1,\ldots , W_n)$.

We will see later that, in order to establish the convergence of both
Stochastic Approximation and Stochastic Gradient Descent, it is enough that
the noise sequence $\{ W_n \}$ satisfies the GSLLN at the rate $\{ \beta_n \}$,
together with some standard assumptions on the map $G$.

\subsection{Stochastic Approximation - Main Result}\label{ssec:SA}

Let $d\ge 1$ be a fixed integer and  $G:\R^d\mapsto \R^d$ be a function
with a unique zero, call it $x^*$.
Thus $G(x^*)=0$ and $G(x)\neq 0$ for all $x\neq x^*$.
We consider a generalization of \eqref{I10}, in which the noise
$W_{n+1}$ is multiplied by a coefficient that can depend on the past history
of the iterations.
The set up is as follows: For $n\ge 0$, let 
$\la_n:(\R^d)^{n+1}\mapsto\R$ be such that, for some constant $C_1$
and some norm $\nmm{\cdot}$ on $\R^d$, $\la_n$ satisfies
\be\label{U5}
|\la_n( u_0, u_1,\ldots, u_n)|\le C_1(1+\max_{0\le k\le n}\nmm{ u_k}).
\ee
Note that, while $\nmm{\cdot}$ can be anything, it is essential to use
the maximum of $\nmm{u_0} , \cdots , \nmm{u_n}$ in \eqref{U5}.
This avoids the norm becoming larger and larger as $n$ increases.
Now the iterative scheme is defined as follows:
For $x_0\in\R^d$ fixed, $X_0=x_0$ and for $n\ge 0$ let
\be\label{U6}
X_{n+1}=X_n-\beta_n \bigl(G(X_n)+\la_n(X_0,\ldots,X_n)W_{n+1}\bigr).
\ee
It is assumed that the function $G:\R^d\mapsto \R^d$ satisfies 
the following condition.
There exist constants $b > 0$, $\rho \in (0,1)$ and
an $x^* \in \R^d$ such that
\be\label{U2}
\nmm{x-x^*- \frac{1}{b}G(x)}\le \rho\nmm{x-x^*} \fa x\in\R^d.
\ee
Our aim is to
establish conditions on the step size sequence $\{\beta_n:n\ge 0\}$
and noise sequence $\{W_n:n\ge 1\}$ that ensure the almost sure 
convergence of $\{X_n:n\ge 0\}$ defined by \eqref{U6} to $x^*$.
Though all norms on $\R^d$ are equivalent, the fact that \eqref{U2}
holds for one particular does not imply that \eqref{U2} holds
for some other norm.
Thus the assumption is that \textit{there exists} some suitable norm
$\nmm{\cdot}$ on $\R^d$ such that \eqref{U2} holds, and it does not matter
which norm it is.
Note that the constants $b$ and $\rho$ appearing in \eqref{U2}
are used only in the proofs, and do not appear in the iteration formula.

Now we present two examples of \eqref{U2}.
\begin{example}\label{exam1}
Suppose $H: \R^d \ap \R^d$ is a contraction with respect to some norm
$\nmm{\cdot}$.
Thus there exists a constant $\rho \in [0,1)$ such that
\bd
\nmm{H(x) - H(y)} \leq \rho \nmm{x - y} , \fa x , y \in \R^d .
\ed
Let $x^*$ denote the unique fixed point of $H$, and define 
$G(x) = x - H(x)$.
Then clearly $x^*$ is the unique solution of $G(x) = \bz$.
Moreover, it is easy to see that \eqref{U2} is satisfied with
$b = 1$.
This is because
\begin{align*}
\nmm{ x - x^* - (x - H(x)) } &= \nmm{ H(x) - x^* } \\
&= \nmm{H(x) - H(x^*) } \leq \rho \nmm{x - x^*} .
\end{align*}
\end{example}

\begin{example}\label{exam2}
Suppose $F: \R^d \ap \R$ is $\C^2$ and strictly convex.
Specifically, let $G$ and $M$ denote $\nabla F$ and $\nabla^2 F$
respectively.
The assumption is that there exist constants $0 < g \leq h < \infty$
such that
\bd
g I_d \leq M(x) \leq h I_d , \fa x \in \R^d ,
\ed
where for symmetric matrices $A$ and $B$, $A \leq B$ is equivalent to
$B-A$ being positive semidefinite.
The above condition is also equivalent to the statement that the
largest eigenvalue of $M(x)$ is bounded above by $h < \infty$,
and the smallest eigenvalue of $M(x)$ is bounded below by $g > 0$.
With these assumptions, it is shown that \eqref{U8} holds with
$\nmm{\cdot} = \nmeu{\cdot}$, and
\bd
b = \frac{1}{1+g+h} , \quad \rho = \frac{1+h}{1+g+h} < 1 .
\ed
The hypotheses guarantee that $F(\cdot)$ has a unique global minimizer
$x^*$, which is also the unique solution to $G(x) = \bz$.
Next, let $x \in \R^d$ be arbitrary, and observe that $G(x^*) = \bz$.
Then
\begin{align*}
G(x) &= \int_0^1 \frac{d}{d \la} G(x^* + \la (x-x^*) ) \; d \la \\
&= \int_0^1 M(x^* + \la (x-x^*) ) \cdot (x - x^*) \; d \la .
\end{align*}
Therefore
\bd
x-x^* - b G(x) = \int_0^1 [ I - b M(x^* + \la (x-x^*)) ] \cdot (x - x^*) \; d \la .
\ed
By assumption, the spectrum of the symmetric matrix $M(x^* + \la (x-x^*))$
lies in the interval $[g,h]$.
Hence the spectrum of
$I - b M(x^* + \la (x-x^*))$ lies in the interval
$[1 - \frac{h}{1+g+h} , 1 - \frac{g}{1+g+h}]$.
Since both numbers are positive, the spectral norm (which is the same
as the $\ell_2$-induced norm) satisfies
\bd
\nmm{I - b M(x^* + \la (x-x^*)) }_{2 \ap 2} \leq 1 - \frac{g}{1+g+h} = \rho .
\ed
Therefore
\bd
\nmm{x-x^* - b G(x)} \le \int_0^1 \rho \nmeu{x - x^*} \; d \la 
= \rho \nmeu{x - x^*} .
\ed

\end{example}

To motivate the concept of the GSLLN,
let us study a special case, when $G(x)=tx$, $t>0$ and $\{\zeta_n:n\ge 0\}$
are bounded $\F_n=\sigma(X_0,\ldots,X_n)$ adapted random variables.
Let $U_0=0$ and for $n\ge 0$, let
\bd
U_{n+1}=(1-t\beta_n )U_n+t\beta_n \zeta_nW_{n+1}, \fa n\ge 0.
\ed
Suppose $\beta_n = 1/(n+1)$, $t=1$ and $\zeta_n=1$ for all $n\ge 0$; then
$U_n$ is just the average of $W_1$ through $W_n$.
Therefore, if each $W_n$ has zero mean, then the almost sure
convergence of $U_n$  to 0 is equivalent to $\{W_n:n\ge 1\}$
satisfying  the strong law of large numbers (SLLN).
It can be seen that if $\{W_n:n\ge 0\}$ are  i.i.d.\
with  mean 0 and finite variance, then \eqref{U8} must hold in order for
$U_n$ to converge to $0$.

With this intuition,
our main result on the convergence of the  Stochastic Approximation
algorithm is the following: 

\begin{theorem}\label{Main-SA}
Suppose $G:\R^d\mapsto \R^d$ satisfies \eqref{U2}.
Let $\la_n:(\R^d)^{n+1}\mapsto\R: n\ge 0$  satisfy \eqref{U5}.
Let $\{\beta_n:n\ge 0\}$ satisfy \eqref{U8} and let 
$\{W_n:n\ge 1\}$ be a sequence of $\R^d$ valued random variables.
Suppose  $\{W_n:n\ge 1\}$ satisfies  the Generalized Strong Law of
Large Numbers (GSLLN) at the rate  $\{\beta_n:n\ge 0\}$
w.r.t.\ the filtration $(\F^W)$. 
Let $x_0\in\R^d$, and define $\{X_n:n\ge 0\}$ by \eqref{U6}. 
Then 
\be\label{U20}
\nmm{X_{n+1}-x^*} \ap  0 \as
\ee
 \end{theorem}
It should be noted that conditions on $G$ and those on the noise
sequence $\{W_n:n\ge 1\}$  are not related to each other.
Of course, the conditions on the rate $\{ \beta_n \}$
and the errors $\{ W_n \}$ are interrelated.

\subsection{Stochastic Gradient Descent - Main Result}\label{ssec:SGD}

In this section, we study the convergence of the Stochastic Gradient
Descent (SGD) algorithm for a $\C^2$-function $F: \R^d \ap \R$.
Most (though not all) papers in this area assume some sort of strong convexity
on $F(\cdot)$, as in Example \ref{exam2}.
However, here we replace that by a different assumption.
Specifically, it is assumed that the equation
$G(x) = \nabla F(x) = \bz$ has a unique solution, denoted by $x^*$,
which is a global minimizer.
For this purpose, we use a generalized version of 
the algorithm described in \eqref{I12}, as described next.

As in the case of the SA algorithm in Section \ref{ssec:SA},
let $\la_n:(\R^d)^{n+1}\mapsto\R: n\ge 0$
denote a ``multiplier'' of the measurement noise terms that depends on
the past iterations.
Specifically, it is assumed that there exists a constant  $C_1$ such that
\be\label{V2}
|\la_n( u_0, u_1,\ldots, u_n)|\le C_1(1+\max_{0\le k\le n}\nmi{u_k}).
\ee
Let $\{\Mp_{m},\Mpp_{m}: m\ge 1\}$ be sequences of $\R^d$ valued
random variables, representing observational errors
appearing in \eqref{V4} below).
Let $\{\psi_{n,i}:n\ge 0, \,1\le i\le d\}$ be  $\{0,1\}$ valued
$(\F^{\Mp,\Mpp})$ adapted random variables.
We will also assume that $\forall i\in \{1,2,\ldots, d\}$
\be\label{V3x}
\sum_{n=0}^\infty\psi_{n,i}\eta_n=\infty \as
\ee
where, as described in Section \ref{ssec:Review},  $\eta_n$ is the step size and $c_n$ that appears below is the increment.
Now we define the iterative procedure.
Choose some $y_0 \in \R^d$, and define $Y_0=y_0$ .
For $n\ge 0$, let  $\xi_n=\la_n(Y_0,\ldots,Y_n)$ and update $Y_n$ as folows:
\be\label{V4}
Y_{n+1}=Y_n- 
\sum_i\psi_{n,i}\eta_n\Bigl(\frac{\bigl(F(Y_n+e_ic_n)+\xi_n
\Mp_{n+1,i}\bigr)-\bigl(F(Y_n-e_ic_n)
+ \xi_n \Mpp_{n+1,i}\bigr)}{2c_n}\Bigr)e_i
\ee
where $e_i$ is the unit vector in the $i^{th}$ direction.

To facilitate the statement of the theorem, we introduce two assumptions:
\ben
\item[(A1).]
The sequences $\{\eta_n:n\ge 0\}\subseteq (0,1)$ and
$\{c_n:n\ge 0\}\subseteq (0,1)$ satisfy
\be\label{V3}
\lim_{n \ap \infty}\eta_n=0, \quad \lim_{n \ap \infty}c_n=0.
\ee
\item[(A2).] It is assumed that there exist constants
$\rho\in (0,1),\, a>0,\,b>0$ such that, for all 
sufficiently small $c$, the following holds:
\be\label{V1} 
|x_i-x_i^*- \frac{1}{2cb}(F(x+ce_i)-F(x-ce_i))|\le
\rho \nmi{ x - x^*} + a c,
\fa 1\le i\le d.
\ee
\een
In the rest of the paper,
the phrase ``for all $n,i$'' will mean ``for all integers $n\ge 0$
and $i\in\{1,2,\ldots,d\}$.''

With these assumptions and notation in place, we can state our main result
on the convergence of the SGD algorithm \eqref{V4}.

\begin{theorem}\label{Main-SGD}
Suppose $F:\R^d\mapsto \R$ satisfies \eqref{V1}.
Let $\la_n:(\R^d)^{n+1}\mapsto\R: n\ge 0$  satisfy \eqref{V2}.
Let $\{\Mp_{n},\Mpp_{n}:n\ge 1\}$ be $\R^d$ valued random variables.
Let  $\{\eta_n, c_n,\psi_{n,i}: n\ge 0, 1\le i\le d\}$ satisfy \eqref{V3}, \eqref{V3x}.
Suppose $\psi_{\cdot,i}$ are adapted to $(\F^{\Mp,\Mpp})$.
Let $\{Y_n:n\ge 0\}$ be defined by \eqref{V4}.
Let $W_{n+1}=\frac{1}{2c_n}(\Mp_{n+1}-\Mpp_{n+1})$ 
for $n\ge 0$.
Suppose  the error sequences $\{W_n:n\ge 1\}$ satisfies Generalized
Strong Law of Large Numbers at the rate  $\{\eta_n:n\ge 0\}$ w.r.t.\
$(\F^{\Mp,\Mpp})$.
Then 
\be\label{W20}
\nmm{Y_{n+1}-x^*} \ap  0 \as
\ee
 \end{theorem}
 
\subsection{Sufficient Conditions for GSLLN}\label{ssec:Suff}

The theorems in the two preceding subsections show the centrality of the
GSLLN property.
In this subsection, we present two sets of sufficient conditions for
a process to satisfy the GSLLN property.
The proofs are deferred to Section \ref{sec:Proofs}.

\begin{theorem}\label{GSLLN}
An $(F_n)$-adapted sequence $\{ W_n \} _{n \geq 0}$ satisfies
the GSLLN at the rate $\{ \beta_n \}$ w.r.t.\ the filtration  $(\F_n)_{n\ge 0}$ if the following properties hold:
for each $n\ge 1$, $i\in\{1,2,\ldots,d\}$,
there exist sets $A_{n,i}\in\F_n$ such that, with
\be\label{G0}
V_{n+1,i}= E[W_{n+1,i}\Ind_{A_{n+1,i}}\mid \F_n]
\ee
we have 
\be\label{G1}
\sum_{n=1}^\infty P(A_{n,i}^c)<\infty,\;\;\;\forall i.
\ee
\be\label{G2}
\sum_{n=0}^\infty\beta_n  E[|V_{n+1,i}|]\;\;\text{converges}\;\;\;\forall i.
\ee
\be\label{G3}
\sum_{n=0}^\infty \beta_n^2 E[|W_{n+1,i}|^2\Ind_{A_{n+1,i}}]<\infty.
\ee
 \end{theorem}

The proof is based on the fact that the hypotheses imply that
\be\label{G7}
\sum_{n=0}^\infty \beta_nW_{n+1,i}\;\;\text{converges}\;\;\;\forall i.
\ee

Next we present an alternate version, which is very similar to the proof
of Kolmogorov's Strong Law of Large numbers;
see \cite[Section 3.7]{Breiman92} for further details.

\begin{theorem}\label{GSLLN2}
An $(F_n)$-adapted sequence $\{ W_n \} _{n \geq 0}$ satisfies
the GSLLN at the rate $\{ \beta_n \}$ w.r.t.\ the filtration  $(\F_n)_{n\ge 0}$  if the following properties hold:
for each $n\ge 1$, $i\in\{1,2,\ldots,d\}$, there exist sets
$A_{n,i}\in\F_n$ such that \eqref{G1} and \eqref{G3} hold; 
further, with $V_{n+1,i}$ defined by \eqref{G0}, we have
\be\label{G2a}
\lim_{n \ap \infty} V_{n+1,i}=0,\;\;\;\text{almost surely}\;\;\;\forall i
\ee
\end{theorem}

Our next result  lists several different conditions, 
each of which ensures  that the GSLLN  holds.
Moreover, under suitable assumptions on $G(\cdot)$, one can combine
each of these assumptions to deduce the convergence of Stochastic Approximation.
 
\begin{theorem}\label{GSLLN-details}
Suppose $\{\beta_n:n\ge 0\}$ and $\{W_n: n\ge 1\}$ satisfy any 
one of the  conditions $(H1)$--$(H5)$ given below, in addition to
$\{\beta_n:n\ge 0\}$ satisfying  \eqref{U8}. 
Then  $\{W_n:n\ge 1\}$ satisfies Generalized Strong Law of Large Numbers
at the rate  $\{\beta_n:n\ge 0\}$ w.r.t.\ the filtration  $(\F^W_n)_{n\ge 0}$.
 \begin{itemize}
\item [$(H1)$] $\{W_n: n\ge 1\}$ are i.i.d. with $ E(\nmm{W_1}^2)<\infty$, $ E(W_1)=0$  and 
\be\label
 {L1} \sum_{n=0}^\infty \beta_n^2<\infty.
 \ee 
\item [$(H2)$] $\{W_n: n\ge 1\}$ are i.i.d.\ with $ E(\nmm{W_1}^\alpha)<\infty$
for some $1\le \alpha <2$, $ E(W_1)=0$, and  there exists a finite constant $D$
such that
\be\label{L2}
 \beta_n\le Dn^{-\frac{1}{\alpha}},\;\;\forall n\ge 1.
\ee 
\item  [$(H3)$] This hypothesis consists of several conditions taken together.
(i) $\{W_n: n\ge 1\}$ are i.i.d.\ with a symmetric distribution
(that is,  $W_1$ and $-W_1$ have the same distribtuion).
 (ii) There exists a constant $0<\delta\le 1$ such that
\be\label{L2a}
 E[\frac{|W_{1,i}|}{({\log(1+|W_{1,i}|)})^\delta}]<\infty , \fa i ,
\ee
(iii) there exists a finite constant D such that $\beta_0\le D$, and
\be\label {L3} \beta_n\le \frac{D}{n({\log(1+n)})^\delta},\;\;\forall n\ge 1.
\ee 
\item  [$(H4)$] Again, this hypothesis consists of several conditions.
(i) $\{W_n: n\ge 1\}$ are independent, (ii) $\mu_n= E(W_n)$ converges to zero,
 (ii) there exists  $\al \in (1,2]$ such that
 $\nu_{\alpha,n} := E[\nmm{W_n}^\al]<\infty$
 for all $n\ge 1$, (iii)   and
\be \label {L4}
 \sum_{n=0}^\infty \beta_n^\alpha\nu_{\alpha,n+1}
= \sum_{n=0}^\infty \beta_n^\al E(\nmm{W_n}^\alpha) <\infty.
\ee 
\item  [$(H5)$] (i) $\{W_n: n\ge 1\}$ is a martingale difference sequence, 
(ii) there exists an $\alpha \in (1,2]$ such that
 $\nu_{\alpha,n}= E(\nmm{W_n}^\alpha)<\infty$ for all $n\ge 1$, and (iii)
\be \label {L5}
 \sum_{n=0}^\infty \beta_n^\alpha\nu_{\alpha,n+1}<\infty.
\ee 
\end{itemize}
\end{theorem}


Our next result  lists a different set of conditions compared to
those in Theorem \ref{GSLLN-details},
each of which ensures  that the GSLLN  holds.
Moreover, under suitable assumptions on $G(\cdot)$, one can combine
each of these assumptions to deduce the convergence of Stochastic 
Gradient Descent.

 \begin{theorem}\label{GSLLN-SGD}
Let $\{\Mp_{n},\Mpp_{n}:n\ge 1\}$ be $\R^d$ valued
random variables, and let  $\{\eta_n, c_n: n\ge 0\}$ satisfy \eqref{V3}.
Define $M_n=\frac{1}{2}(\Mp_{n} - \Mpp_{n})$, and
$W_{n+1}=\frac{1}{c_n}M_{n+1}$.

Suppose that one of the  conditions $(K1)$--$(K5)$ given below holds. Then GSLLN holds for $\{W_{n+1}:n\ge 0\}$ at the rate $\{\eta_n:n\ge 0\}$ w.r.t.\
$(\F^{\Mp,\Mpp})$,  and as a consequence, 
 Stochastic Gradient Descent Theorem \ref{Main-SGD} holds (if in addition, \eqref{V1}, \eqref{V2}, \eqref{V3},  \eqref{V3x} hold).
 
 \begin{itemize}
\item [$(K1)$]   Suppose $\{\Mp_{n},\Mpp_{n}:n\ge 1\}$ are i.i.d. with $ E(\nmm{\Mp_1}^2)<\infty$  and 
\be \label {N1} \sum_{n=0}^\infty \frac{\eta_n^2}{c_n^2}<\infty.
\ee 
 \item  [$(K2)$]Suppose $\{\Mp_{n},\Mpp_{n}:n\ge 1\}$ are i.i.d.,   for some $\alpha$, $1\le \alpha< 2$, $ E(\nmm{\Mp_1}^\alpha)<\infty$, and  for some finite constant D, $\beta_0\le D$ and
\be \label {N2}  \frac{\eta_n}{c_n}\le Dn^{-\frac{1}{\alpha}},\;\;\forall n\ge 1.
\ee 
\item  [$(K3)$] Suppose $\{\Mp_{n},\Mpp_{n}:n\ge 1\}$ are i.i.d.,  for some $0<\delta\le 1$, $ E[\frac{|\Mp_{1,i}|}{(\log(1+|\Mp_{1,i}|))^\delta}]<\infty$ for each $i$ and  for some finite constant D, $\beta_0\le D$ and
\be \label {N3}  \frac{\eta_n}{c_n}\le \frac{D}{n(\log(1+n))^\delta},\;\;\forall n\ge 1.
\ee
\item  [$(K4)$] $\{M_n: n\ge 1\}$ are independent, $\nu_n= E(M_n)$ is such that $\frac{\nu_{n+1}}{c_n}$ converges to zero, for some $\alpha$,  $1<\alpha\le 2$, $\nu_{\alpha,n}= E(\nmm{M_n}^\alpha)<\infty$ for all $n\ge 1$    and
\be \label {N4} \sum_{n=0}^\infty \frac{\eta_n^\alpha}{c_n^\alpha}\nu_{\alpha,n+1}<\infty.
\ee 
\item  [$(K5)$] $\{M_n: n\ge 1\}$ is a martingale difference sequence, for some $\alpha$, $1<\alpha\le 2$,  $\nu_{\alpha,n}= E(\nmm{M_n}^\alpha)<\infty$  for all $n\ge 1$    and
\be \label {N5} \sum_{n=0}^\infty \frac{\eta_n^\alpha}{c_n^\alpha}\nu_{\alpha,n+1}<\infty.
\ee 
 \end{itemize}
 \end{theorem}

\section{Proofs}\label{sec:Proofs}
  
We begin with
an auxiliary result, which plays an important role in our approach.
The second part of this can be thought of as an extension of Kronecker's
lemma, as stated in \cite[Lemma 3.28]{Breiman92}.
\begin{lemma}\label{deterministic}
Let $\{\tau_n:n\ge 0\}\subseteq (0,\infty)$ be such that 
\be\label{B0} \lim_{n\rightarrow\infty} \tau_n=0,\;\;\; \sum_{n=0}^\infty \tau_n=\infty.\ee
Let $\{z_n:n\ge 1\} \subseteq \R^d$ and let $\{s_n:n\ge 0\} \subseteq \R^d$ be defined by $s_0=0$ and
\be\label{B1}
s_{n+1}=(1-\tau_n)s_n+\tau_nz_{n+1}.\ee
Then we have
\ben
\item[(i)] If $\lim_n\nmm{z_{n+1}}=0$, then $\lim_n\nmm{s_n}=0$.
\item[(ii)] If $\sum_{k=0}^n\tau_kz_{k+1}$ converges in $\R^d$,
then  also $\lim_n\nmm{s_n}=0$.
\een
\end{lemma}

\begin{proof}
Let us also note that in view of \eqref{B0}, there exists an
$n^*$ such that $\tau_j<1$ for all $j\ge n^*$. 
 We can check using induction on $n$ that for $n^*\le m\le n$ we have
\be  \label{B2}
  \bigl[\prod_{\{m\le j\le n\}}(1-\tau_j)\bigr]+\sum_{k=m}^n\bigl[\prod_{\{k<j\le n\}}(1-\tau_j)\bigr]\tau_k=1,\ee
\be  \label{B2a}
  \bigl[\prod_{\{m\le j\le n\}}(1-\tau_j)\bigr]s_m
+\sum_{k=m}^n\bigl[\prod_{\{k<j\le n\}}(1-\tau_j)\bigr]\tau_kz_{k+1}=  {s_{n+1}}.\ee
Using \eqref{B2} we can see that for $n^*\le m\le n$ we have
\be  \label{B3}
 \sum_{k=m}^n\bigl[\prod_{\{k<j\le n\}}(1-\tau_j)\bigr]\tau_k=1-\bigl[\prod_{\{m\le j\le n\}}(1-\tau_j)\bigr]\le 1.
\ee
 If $\lim_n\nmm{z_{n+1}}=0$ then given $\epsilon>0$, we can choose $m> n^*$ such that 
 $\nmm{z_{n+1}}<\epsilon$ for all $n\ge m$.
Then for $n\ge m$ we have (in view of \eqref{B3})
 \begin{align*} 
 \nmm{s_{n+1}}&\le   \bigl[\prod_{\{m\le j\le n\}}(1-\tau_j)\bigr]\nmm{s_m}+\sum_{k=m}^n\bigl[\prod_{\{k<j\le n\}}(1-\tau_j)\bigr]\tau_k\epsilon\\
 &\le    \bigl[\prod_{\{m\le j\le n\}}(1-\tau_j)\bigr]\nmm{s_m}+\epsilon.
 & \end{align*}
 Taking limit as $n \ap  \infty$, we conclude that $\lim_n\nmm{s_{n+1}}\le \epsilon$.  Since this holds for all $\epsilon>0$, we conclude that
 $\nmm{s_n} \ap  0$ completing proof of $(i)$. 
 
 Coming to the proof of (ii), 
for $0\le m\le n$, let 
 \be\label{J31}
 q_{m,n}= {\sum_{j=m}^n}\tau_jz_{j+1},\ee
 In view of assumption that $ {\sum_{j=1}^n}\tau_jz_{j+1}$ converges as $n \ap \infty$, it follows that $\forall \epsilon>0$, $\exists$ $n_\epsilon$ such that 
\be\label{B11}
\nmm{q_{m,n}}\le\epsilon\;\;\;\forall n\ge m\ge n_\epsilon.
\ee
Now writing ${c}_{n+1,n}=0$, it follows using \eqref{B2a} that
\bd
\begin{split}
{s_{n+1}}=&  \bigl[\prod_{\{m\le j\le n\}}(1-\tau_j)\bigr]s_m+\sum_{k=m}^n\bigl[\prod_{\{k<j\le n\}}(1-\tau_j)\bigr]\tau_kz_{k+1}.\\
=& \bigl[\prod_{\{m\le j\le n\}}(1-\tau_j)\bigr]s_m+ {\sum_{k=m}^n}[\prod_{\{k<j\le n\}}(1-\tau_j)](q_{k,n}-q_{k+1,n})\\
=& \bigl[\prod_{\{m\le j\le n\}}(1-\tau_j)\bigr]s_m+ [\prod_{\{m<j\le n\}}(1-\tau_j)]q_{m,n}\\
&+ {\sum_{t=m+1}^n}[\prod_{\{t<j\le n\}}(1-\tau_j)-\prod_{\{t-1<j\le n\}}(1-\tau_j)]
q_{t,n} \\
\end{split}
\ed
Thus, in view of \eqref{B11}, we have that, for all $n\ge m\ge n_\epsilon$,
\be
\begin{split}
\nmm{s_{n+1}}\le& \bigl[\prod_{\{m\le j\le n\}}(1-\tau_j)\bigr]\nmm{s_m}+ [\prod_{\{m<j\le n\}}(1-\tau_j)]\nmm{q_{m,n}}\\
&+ {\sum_{t=m+1}^n}[\prod_{\{t<j\le n\}}(1-\tau_j)-\prod_{\{t-1<j\le n\}}(1-\tau_j)]
\nmm{q_{t,n}}\\
 \le &  \bigl[\prod_{\{m\le j\le n\}}(1-\tau_j)\bigr]\nmm{s_m}+ [\prod_{\{m<j\le n\}}(1-\tau_j)]\epsilon\\
&+ {\sum_{t=m+1}^n}[\prod_{\{t<j\le n\}}(1-\tau_j)-\prod_{\{t-1<j\le n\}}(1-\tau_j)]
\epsilon\\
\le&   \bigl[\prod_{\{m\le j\le n\}}(1-\tau_j)\bigr]\nmm{s_m}+ \epsilon.
\end{split}
\ee
As seen in proof of  (i), this yields $\nmm{s_n} \ap  0$. 
\end{proof}

We will need another observation to consider the case of 
i.i.d.\ errors when the mean may not be defined.
\begin{lemma}\label{Log}
$\forall x>0$, $0\le \delta\le 1$ 
\be\label{b49}  
\frac{x}{(\log(1+x))^\delta}\le y\;\;\Longrightarrow \;\;{\log(1+x)}\le 3\,{{\log(1+y)}}
\ee
\end{lemma}

\begin{proof}
Let us note that for all $x>0$, $(1+x)\le \exp(\sqrt{2x}\,)$ as can be easily seen. Hence $\log(1+x)\le \sqrt{2x}$ and thus 
$ \frac{x}{\log(1+x)}\ge \frac{x}{\sqrt{2x}}=\frac{\sqrt{x}}{\sqrt{2}}.$ 

If $\frac{x}{{\log(1+x)}}\le y$ then $(1+\frac{\sqrt{x}}{\sqrt{2}})\le (1+y)$. Since $(1+\frac{\sqrt{x}}{\sqrt{2}}\,)^3\ge (1+x)$, we have $(1+x)\le (1+y)^3$ from which the required conclusion follows by taking logrithm. 
This proves the case $\delta=1$. If $\delta <1$, 
\[\frac{x}{(\log(1+x))^\delta}\ge \frac{x}{(\log(1+x))}\]
and hence the result follows from the case $\delta=1$.
\end{proof}

We now come to the proof of Theorem \ref{GSLLN}. 

\begin{proof}
\textbf{Of Theorem \ref{GSLLN}.}
Let us fix  a  $(\F_n)$ adapted sequence of random variables $\{\zeta_n:n\ge 1\}$ that is bounded, say by a constant $C$
and let
\bd
S_{n+1}=(1-t\beta_n )S_n+t\beta_n \zeta_n W_{n+1},\;\;n\ge 0.
\ed
Let us write $W_{n+1,i}=U_{n+1,i}+V_{n+1,i}+R_{n+1,i}$ where
\begin{align*}
U_{n+1,i}&=W_{n+1,i}\Ind_{A_{n+1,i}^c},\\
V_{n+1,i}&= E[W_{n+1,i}\Ind_{A_{n+1,i}}\mid \F_n],\\
R_{n+1,i}&=W_{n+1,i}\Ind_{A_{n+1,i}}-V_{n+1,i}.\end{align*}
Let $H_0=0$, $J_0=0$ and $K_0=0$
for $n\ge 0$ 
\begin{align*}
H_{n+1,i}&=(1-t\beta_n )H_n+t\beta_n \zeta_n U_{n+1,i}\\
J_{n+1,i}&=(1-t\beta_n )J_n+t\beta_n \zeta_n V_{n+1,i}\\
K_{n+1,i}&=(1-t\beta_n )K_n+t\beta_n \zeta_n R_{n+1,i}.\end{align*}
Clearly, $S_{n+1,i}=H_{n+1,i}+J_{n+1,i}+K_{n+1,i}$;
now we will prove that each of $H_{n+1,i}$, $J_{n+1,i}$ and $K_{n+1,i}$
converge to 0 almost surely as $n \ap \infty$, $\forall i$.
That will complete the proof of the theorem. 

In view of condition \eqref{G1} and the Borel-Cantelli lemma we 
can conclude that $\zeta_nU_{n+1,i}$ converges to zero almost surely. 
Hence in view of Lemma \ref{deterministic} it follows that
$H_{n+1,i}$ converge to zero almost surely.
If $V_{n+1,i}$ satisfies \eqref{G2a} then so does  $\zeta_nV_{n+1,i}$, because $\{\zeta_n\}$ is bounded.
As a consequence, $J_{n+1,i}$ converges to 0 in view of (i) in
Lemma \ref{deterministic} while if $V_{n+1,i}$ satisfies \eqref{G2} then
again so does  $\zeta_nV_{n+1,i}$.
As a consequence
 $J_{n+1,i}$ converges to 0 in view of (ii) in Lemma \ref{deterministic}.
 It is clear that
\begin{align*}
 E[|R_{n+1,i}|^2\mid \F_n]&= E[|W_{n+1,i}|^2\Ind_{A_{n+1,i}}\mid  \F_n]-|V_{n+1,i}|^2\\
&\le  E[|W_{n+1,i}|^2\Ind_{A_{n+1,i}}\mid  \F_n]\\
\end{align*}
and since $\{\zeta_n:n\ge 1\}$ is  $(\F_n)$ adapted and is bounded (by $C$), we have 
\be\label{B21}
 E[|\zeta_nR_{n+1,i}|^2]\le C^2  E[|R_{n+1,i}|^2]\le C^2 E[|W_{n+1,i}|^2\Ind_{A_{n+1,i}}].
\ee
Thus $Q_{n+1,i}=\sum_{k=0}^n\beta_k\zeta_nR_{k+1,i}$ is an $L^2$-bounded  martingale in view of assumption \eqref{G3}.
Therefore it follows that
\be\label{B21x}
 \sum_{k=0}^n\beta_k\zeta_nR_{k+1,i} \;\text{ converges almost surely.}
\ee 
 Invoking part (ii)  of Lemma \ref{deterministic} it follows that $K_{n+1,i}$ converges to zero almost surely, completing the proof as remarked earlier.
\end{proof}

We now proceed with the proof of Theorem \ref{Main-SA}.

\begin{proof}
\textbf{Of Theorem \ref{Main-SA}.}
Let $\xi_n=\la_n( X_0, X_1,\ldots, X_n)$ where  $\la_n$ satisfies \eqref{U5} and  let $\phi_n=(1+\max_{0\le k\le n}\nmm{X_k})$. Let 
$b$ be such that  \eqref{U2} holds and let
 $\zeta_n=-b^{-1}\phi_n^{-1}\xi_n$,  $Z_{n+1}=\zeta_nW_{n+1}$,  $\tau_n=b\beta_n$ and $R_n=b^{-1}G(X_n)$.  Let us note that $\zeta_n$ is bounded by $b^{-1}C_1$. Since $\tau_n \phi_nZ_{n+1}=b\beta_n\phi_n\zeta_nW_{n+1}= -\beta_n\xi_nW_{n+1}$
 and as a result we can 
 rewrite \eqref{U6} as
\begin{align*}
X_{n+1}&=X_n-\beta_n \bigl(G(X_n)+\xi_nW_{n+1}\bigr)\\
&=X_n-\tau_n b^{-1}G(X_n)-\beta_n\xi_nW_{n+1}\\
&=X_n-\tau_n R_n+\tau_n\phi_nZ_{n+1}\end{align*}
\vskip -8mm
and hence
\begin{equation}\label{B31}
X_{n+1}-x^*=(1-\tau_n)(X_n-x^*)+\tau_n((X_n-x^*)-R_n)+\tau_n\phi_nZ_{n+1}.
\end{equation}
Note that in view of the assumption \eqref{U2} we have
\begin{equation}\label{B32}
\nmm{((X_k-x^*)-R_k)}\le \rho\nmm{X_k-x^*}.
\end{equation}
Before we proceed, let us note that $\phi_n$ is an increasing process.
Since $\{W_n:n\ge 1\}$ satisfies GSLLN at the rate  $\{\beta_n:n\ge 0\}$, and $\zeta_n$ is bounded, $\tau_n=b\beta_n$ it follows that $S_{n+1}$ converges to zero almost surely where 
 $\{S_n:n\ge 0\}$ is defined by  $S_0=0$ and 
\begin{equation}\label{B28}
S_{n+1}=(1-\tau_n)S_n+\tau_n Z_{n+1},\;\;n\ge 0.\end{equation}
 We are going to show that for every $\omega \in \Omega$ such that $S_{n+1}(\omega)$ defined by \eqref{B28} converges, $X_{n+1}(\omega)-x^*$ defined by  \eqref{B31} converges to zero. We will drop $\omega$ in the notation as usual, and we should keep in mind that the choices of $n,\,m,\,K$ in what follows can depend upon $\omega$. 
 
Let $m^*$ be chosen such that $\tau_n=b\beta_n<1$ for all $n\ge m^*$. Since $\{\beta_n:n\ge 0\}$ satisfies  \eqref{U8}, so does  $\{\tau_n:n\ge 0\}$. 
We can check, using recursion and \eqref{B31}
that for $ 0\le m\le n$,  we have
 \begin{equation}\label{R8}
 \begin{split}
X_{n+1}-x^*=& \bigl[\prod_{j=m}^{n}(1-\tau_j)\bigr](X_m-x^*)\\&\;\;\;\;+ \sum_{k=m}^{n} \bigl[\prod_{_{\{j:k<j\le n\}}}(1-\tau_j)\bigr]\tau_k((X_k-x^*)-R_k)\\
&\;\;\;\;+ \sum_{k=m}^{n} \bigl[\prod_{_{\{j:k<j\le n\}}}(1-\tau_j)\bigr]\phi_k\tau_k Z_{k+1}. 
\end{split}\end{equation}
Here, product over an empty set is taken as 1.
 Let us introduce notation for $0\le m\le n$,
 \begin{align*}
  U_{m,n}&= \bigl[\prod_{j=m}^{n}(1-\tau_j)\bigr](X_m-x^*)\\
  V_{m,n}&= \sum_{k=m}^{n} \bigl[\prod_{_{\{j:k<j\le n\}}}(1-\tau_j)\bigr]\tau_k((X_k-x^*)-R_k)\\
 L_{m,n}&= {\sum_{k=m}^n}\bigl[\prod_{\{k<j\le n\}}(1-\tau_j)\bigr]\phi_k\tau_k Z_{k+1}\\
  D_{m,n}&= {\sum_{k=m}^n}\bigl[\prod_{\{k<j\le n\}}(1-\tau_j)\bigr]\tau_k Z_{k+1}
\end{align*}
Let us note (by recursion on $n$, as in \eqref{B2a}) that
 \begin{equation}\label{B41}
 S_{n+1}=  \bigl[\prod_{\{m\le j\le n\}}(1-\tau_j)\bigr]S_m
+D_{m,n}.\ee
In view of \eqref{R8}, we have for any $0\le m\le n$
\[X_{n+1}-x^*= U_{m,n}+ V_{m,n}+ L_{m,n}\]
and hence (noting that $m^*\le m\le n$ implies $(1-\tau_j)>0$ for all $j$, $m\le j\le n$)
 \begin{equation}\label{R19}
\nmm{X_{n+1}-x^*}\le \nmm{ U_{m,n}}+\nmm{ V_{m,n}}+\nmm{ L_{m,n}},\;\;\;\forall m^*\le m\le n.
\end{equation}
Let $\gamma_n=\max\{\nmm{X_k-x^*},\;\;0\le k\le n\}$ and $\gamma^*=\sup_n \gamma_n$. Then $\nmm{X_k}\le (\gamma_n+\nmm{x^*})$ for $0\le k\le n$, and recalling that $\phi_n=(1+\max_{0\le k\le n}\nmm{ X_k})$, it follows that
\begin{equation}\label{R25}
\phi_n\le (1+\gamma_n+\nmm{x^*}),\;\;\;\;\forall n.\end{equation}
We will first prove that for all
$\epsilon>0$, $\exists$ $n_\epsilon $ such that 
\begin{equation}\label{R24}
\nmm{ L_{m,n}}\le(1+\gamma_n+\nmm{x^*})\epsilon\;\;\;\;\;\forall n\ge m\ge n_\epsilon.
\end{equation}
Recall, we are working with $\omega$ such that $S_{n+1}=S_{n+1}(\omega)$ converges to zero and hence
given $\epsilon>0$,   we can choose $n_\epsilon>m^*$ such that 
\[
\nmm{S_{n+1}}\le  \frac{1}{2}\epsilon\;\;\;\forall n\ge n_\epsilon.
\]
Now in view of  \eqref{B41}, it follows that 
\begin{equation}\label{J11}
\nmm{D_{m,n+1}}\le  \epsilon\;\;\;\forall n\ge m\ge n_\epsilon.
\end{equation} 
Let us note that
\be\label{R11}
\begin{split}
 L_{m,n}=& {\sum_{k=m}^n}\phi_k[\prod_{\{k<j\le n\}}(1-\tau_j)]\tau_k Z_{k+1}\\
=&\phi_m {\sum_{k=m}^n}[\prod_{\{k<j\le n\}}(1-\tau_j)]\tau_k Z_{k+1}\\
&\;\;\;\;\;\;\;\;\;\;\;\;\;\;\;\;+ {\sum_{k=m+1}^n}(\phi_k-\phi_m)[\prod_{\{k<j\le n\}}(1-\tau_j)]\tau_k Z_{k+1}\\
=&\phi_mD_{m,n}+ {\sum_{k=m+1}^n}(\phi_k-\phi_m)[\prod_{\{k<j\le n\}}(1-\tau_j)]\tau_k Z_{k+1}\\
\end{split}\ee
Proceeding likewise, we will get
\[ L_{m,n}=\phi_mD_{m,n}+(\phi_{m+1}-\phi_m)D_{m+1,n}+\ldots +(\phi_{n}-\phi_{n-1})D_{n,n}\]
Since $\phi_k$ is increasing, in view of \eqref{J11} it follows that $\forall n\ge m\ge n_\epsilon$, 
\[\begin{split} \nmm{ L_{m,n}}&\le \phi_m\nmm{D_{m,n}}+(\phi_{m+1}-\phi_m)\nmm{D_{m+1,n}}+\ldots +(\phi_{n}-\phi_{n-1})\nmm{D_{n,n}}\\
&\le \phi_m  \epsilon+(\phi_{m+1}-\phi_m)  \epsilon+\ldots +(\phi_{n}-\phi_{n-1})  \epsilon\\
&=\phi_{n}  \epsilon\\
&\le(1+\gamma_n+\nmm{x^*})\epsilon\end{split}
\]
where we have used \eqref{R25} at the last step. 
This proves \eqref{R24}. Next we will prove that
\begin{equation}\label{b21}
\gamma^*=\sup\{\nmm{X_n-x^*}:n\ge 0\}<\infty.
\end{equation}
 First, using $\rho<1$, get $\epsilon>0$   such that  $\epsilon<\frac{1}{3}$ and
\begin{equation}\label{R14}
\rho(1+3\epsilon)\le 1.
\end{equation}
For this $\epsilon$, choose $n_\epsilon$ such that \eqref{R24} holds.  Now we will prove, by induction that $\forall k\ge n_\epsilon$
\begin{equation}\label{R13}
\gamma_k\le (1+3\epsilon)(1+\gamma_{n_\epsilon}+\nmm{x^*}).
\end{equation}
Clearly \eqref{R13} is true for $k=n_\epsilon$.   
Suppose that \eqref{R13} holds for all $k$ such that  $n_\epsilon\le k\le n^*$.
We will prove that \eqref{R13} also holds for $k=n^*+1$, completing the proof that \eqref{R13} is true for all $k\ge n_\epsilon$.  To prove the induction step, suffices to prove  that
\[\nmm{X_{n^*+1}-x^*}\le (1+3\epsilon)(1+\gamma_{n_\epsilon}+\nmm{x^*}).\]
For $n_\epsilon\le k\le n^*$, \eqref{B32}, \eqref{R14}, \eqref{R13} imply that
\begin{equation}\label{b25a} \begin{split}
\nmm{((X_k-x^*)-R_k)}&\le \rho\nmm{X_k-x^*}\\
&\le \rho\gamma_k\\
&\le \rho (1+3\epsilon)(1+\gamma_{n_\epsilon}+\nmm{x^*})\\
&\le (1+\gamma_{n_\epsilon}+\nmm{x^*}).
\end{split}
\end{equation}
As a consequence 
\begin{equation}\label{R15}\begin{split}
\nmm{ V_{n_\epsilon,n^*}}&\le  \sum_{k=n_\epsilon}^{n^*} \bigl[\prod_{_{\{j:k<j\le n\}}}(1-\tau_j)\bigr]\tau_k\nmm{(  (X_k-x^*)-R_k)}\\
&\le (1+\gamma_{n_\epsilon}+\nmm{x^*}) \sum_{k=n_\epsilon}^{n^*} \bigl[\prod_{_{\{j:k<j\le n^*\}}}(1-\tau_j)\bigr]\tau_k.
\end{split}\end{equation}
It is easy to see that
\begin{equation}\label{b25b}\begin{split}
 \nmm{ U_{n_\epsilon,n^*}}&=  \bigl[\prod_{j=n_\epsilon}^{n^*}(1-\tau_j)\bigr]\nmm{X_{n_\epsilon}-x^*}\\
   &\le \gamma_{n_\epsilon} \bigl[\prod_{j=n_\epsilon}^{n^*}(1-\tau_j)\bigr].\end{split}
 \end{equation}
For all $m\le n$ one can verify (by induction on $n$) that
\begin{equation}\label{R17}
 \bigl[\prod_{k=m}^{n}(1-\tau_k)\bigr]+ \sum_{k=m}^{n} \bigl[\prod_{_{\{j:k<j\le n\}}}(1-\tau_j)\bigr]\tau_k=1,    
\end{equation}
and then putting together \eqref{R15}, \eqref{b25b}, \eqref{R17} we conclude that
\begin{equation}\label{R20}
\nmm{ U_{n_\epsilon,n^*}}+\nmm{ V_{n_\epsilon,n^*}}\le (1+\gamma_{n_\epsilon}+\nmm{x^*}).\end{equation}
Using \eqref{R24}, along with \eqref{R13} for $k=n^*$ we deduce
  \begin{equation}\label{R21}
\begin{split}
\nmm{ L_{n_\epsilon,n^*}}
&\le \epsilon(1+\gamma_{n^*}+\nmm{x^*})\\
&\le \epsilon(1+\nmm{x^*})+ \epsilon\gamma_{n^*}\\
&\le \epsilon(1+\nmm{x^*})+ \epsilon (1+3\epsilon)(1+\gamma_{n_\epsilon}+\nmm{x^*})\\
&\le (2\epsilon+3\epsilon^2)(1+\nmm{x^*})+\epsilon (1+3\epsilon)\gamma_{n_\epsilon}\\
&\le 3\epsilon(1+\gamma_{n_\epsilon}+\nmm{x^*}). 
\end{split}
\end{equation}
Here, since $\epsilon<\frac{1}{3}$, we have used $(2\epsilon+3\epsilon^2)\le 3\epsilon$ and also $(1+3\epsilon)\le 2$.
We thus conclude, using \eqref{R19}, \eqref{R20} along with \eqref{R21}, that
\begin{equation}\label{b31}\begin{split}
\nmm{X_{n^*+1}-x^*}&\le \nmm{ U_{n_\epsilon,n^*}}+\nmm{ V_{n_\epsilon,n^*}}+\nmm{ L_{n_\epsilon,n^*}}\\
&\le (1+3\epsilon)(1+\gamma_{n_\epsilon}+\nmm{x^*}).
\end{split}\end{equation}
This completes the induction proof and thus \eqref{R13} holds for all $k\ge n_{\epsilon}$  proving \eqref{b21}.

It also follows that $\sup\{\nmm{X_k}: k\ge 0\}<\infty $ and hence that 
\begin{equation}\label{b51}
\sup\{\phi_{n}:n\ge 0\}<\infty.
\end{equation}
Using \eqref{R24} and \eqref{b51}, it follows that for all $\delta>0$, we can get $m_\delta$ such that
 \begin{equation}\label{b34}  
  \nmm{ L_{m,n}}<\delta \;\;\;\forall m,n \text{ such that }m_\delta\le m\le n
  \end{equation}
and thus we have for $m\ge m_\delta$
\begin{equation}\label{b35}
\limsup_{n \ap \infty}\nmm{ L_{m,n}}\le\delta.
\end{equation}
Note that in view of the assumption that $\{\beta_n:n\ge 0\}$ satisfies \eqref{U8} and $\tau_n=b\beta_n$, we have $\sum_n\tau_n=\infty$ and hence
\begin{equation}\label{b32}
    \lim_{n \ap \infty}\bigl[\prod_{j=m}^{n}(1-\tau_j)\bigr]=0, \;\;\;\forall m<\infty.
  \end{equation} 
For $m\ge 1$, let $\theta_m=\sup\{\nmm{X_k-x^*}: k\ge m\}$. Then clearly, $\theta_m$ is decreasing and $\theta_m\le \gamma^*<\infty$ for all $m$. Let $\theta^*=\lim_{m \ap \infty}\theta_m$. Clearly $\theta^*\le \gamma^*$. Remains to show that $\theta^*=0$.
Since
\[ \nmm{ U_{m,n}}\le \    =\nmm{X_m-x^*} \bigl[\prod_{j=m}^{n}(1-\tau_j)\bigr]\]
and thus in view of \eqref{b32} we have
\begin{equation}\label{b33}
\lim_{n \ap \infty}\nmm{ U_{m,n}}=0 \;\;\;\forall m<\infty.
  \end{equation} 
Note that for any $m\le n$ we have
\begin{equation}\label{b37}\begin{split}
\nmm{ V_{m,n}}&\le  \sum_{k=m}^{n} \bigl[\prod_{_{\{j:k<j\le n\}}}(1-\tau_j)\bigr]\tau_k\nmm{(  (X_k-x^*)-R_k)}\\
&\le  \sum_{k=m}^{n} \bigl[\prod_{_{\{j:k<j\le n\}}}(1-\tau_j)\bigr]\tau_k\rho\nmm{(X_k-x^*)}\\
&\le\rho \sum_{k=m}^{n} \bigl[\prod_{_{\{j:k<j\le n\}}}(1-\tau_j)\bigr]\tau_k\theta_k \\
&\le \rho\theta_m \\
\end{split}\end{equation}   
Hence for all $m$
\begin{equation}\label{R27}
\limsup_{n \ap \infty}\nmm{ V_{m,n}}\le 
\rho\theta_m.\end{equation}
Combining \eqref{R19} along with \eqref{b35}, \eqref{b33} and  \eqref{R27}, we get for any $m\ge m_\delta$ 
\[
\limsup_{n \ap \infty}\nmm{X_{n+1}-x^*}\le \rho\theta_m +\delta.\]
Now taking limit as $m \ap \infty$ on the RHS above and then limit as $\delta\downarrow 0$, we get
\[\limsup_{n \ap \infty}\nmm{X_{n+1}-x^*}\le \rho\theta^*.\]
From the definition of $\theta^*$, it follows that  $\limsup_{n \ap \infty}\nmm{X_{n+1}-x^*}=\theta^*$ and hence we have shown
\begin{equation}\label{b43}  
\theta^*\le \rho\theta^*.
\end{equation} 
Since $\rho<1$ and $0\le\theta^*<\infty$, \eqref{b43} implies $\theta^*=0$. This completes proof of Theorem   \ref{Main-SA}.
\end{proof}

Next we present a proof of Theorem \ref{Main-SGD}.

\begin{proof}
\textbf{Of Theorem \ref{Main-SGD}.}
Several steps in the proof are the same as in the proof of
Theorem \ref{Main-SA}. 
Recall that $\xi_n=\la_n( Y_0, Y_1,\ldots, Y_n)$ where  $\la_n$ satisfies \eqref{V2} and  let $\phi_n=(1+\max_{0\le k\le n}\nmi{ Y_k})$.  Let 
$b$ be such that  \eqref{V1} holds  and let
 $\zeta_n=-b^{-1}\phi_n^{-1}\xi_n$, $T_{n,i}=  \frac{1}{2bc_n}\bigl(F(Y_n+e_ic_n)-F(Y_n-e_ic_n))$, $Z_{n+1,i}=\psi_{n,i}\zeta_nW_{n+1,i}$,  $\tau_n=b\eta_n$. 
Then $\tau_n \phi_nZ_{n+1,i}=b\eta_n\psi_{n,i}\phi_n\zeta_nW_{n+1,i}= -\eta_n\psi_{n,i}\xi_nM_{n+1,i}$. Note that $\{\zeta_n:n\ge 0\}$ is bounded. 
For later use, let us note that in view of assumption  \eqref{V1} on $F$, it follows that
\be\label{f2}  
|(Y_{n,i}-x^*_i)-T_{n,i}|\le \rho |Y_{n,i}-x^*_i|+ ac_n.
\ee

Recalling that  $W_{n+1}=\frac{1}{2c_n}(\Mp_{n+1}-\Mpp_{n+1})$, $n\ge 0$, let us rewrite \eqref{V4} as follows:
\begin{align*}
Y_{n+1}-x^*=&Y_n-x^*- 
\sum_i\psi_{n,i}\eta_n\Bigl(\frac{\bigl(F(Y_n+e_ic_n)+\xi_n
\Mp_{n+1,i}\bigr)-\bigl(F(Y_n-e_ic_n)+\xi_n\Mpp_{n+1,i}\bigr)}{2c_n}\Bigr)e_i\\
=&Y_n-x^*- \sum_i\psi_{n,i}\tau_n\bigl(\frac{F(Y_n+e_ic_n)-F(Y_n-e_ic_n)}{2bc_n}\bigr)e_i- \sum_i\psi_{n,i}\eta_n\bigl(\frac{\xi_n(
\Mp_{n+1,i}-\Mpp_{n+1,i})}{2c_n}\bigr)e_i\\
=& \sum_i(1-\psi_{n,i}\tau_n)(Y_{n,i}-x^*_i)e_i\\
&+ \sum_i\psi_{n,i}\tau_n\bigl(Y_{n,i}-x^*_i- \frac{1}{2bc_n}\bigl(F(Y_n+e_ic_n)-F(Y_n-e_ic_n)\bigr)\bigr)e_i-\sum_i\psi_{n,i}\eta_n \xi_nW_{n+1}e_i\\
=& \sum_i\bigl((1-\psi_{n,i}\tau_n)(Y_{n,i}-x^*_i)+\psi_{n,i}\tau_n((Y_{n,i}-x^*_i)-T_{n,i})+\tau_n\phi_nZ_{n+1,i}\bigr)e_i\end{align*}
and hence
\begin{equation}\label{C1}
Y_{n+1,i}-x^*_i=(1-\psi_{n,i}\tau_n)((Y_{n,i}-x^*_i)+\psi_{n,i}\tau_n(Y_{n,i}-x^*_i)-T_{n,i})+\tau_n\phi_nZ_{n+1,i}.
\end{equation}
We we will now prove that $\nmi{Y_{n+1} - x^*} \ap 0$ almost surely.

Before we proceed, let us note that $\phi_n$ is an increasing process and also that when $\psi_{n,i}=0$, $Z_{n+1,i}=0$ and $Y_{n+1,i}-x^*_i=Y_{n,i}-x^*_i$. 
Since $\{W_n:n\ge 1\}$ satisfies GSLLN at the rate  $\{\eta_n:n\ge 0\}$, and $\zeta_n$ is bounded, $\tau_n=b\eta_n$ it follows that $S_{n+1,i}$ converges to zero almost surely for each $i$, $1\le i\le d$ where 
 $\{S_{n,i}:n\ge 0\}$ is defined by  $S_{0,i}=0$ and 
\begin{equation}\label{C2}
S_{n+1,i}=(1-\tau_n)S_{n,i}+\tau_n Z_{n+1,i},\;\;\;n\ge 0.\end{equation}
Recall that $\{\eta_n, c_n,\psi_{n,i}: n\ge 0, 1\le i\le d\}$ satisfy \eqref{V3}, \eqref{V3x}. Thus $\sum_n\psi_{n,i}\tau_n=\infty$ almost surely. 
 We are going to show that for every $\omega \in \Omega$ such that $S_{n+1,i}(\omega)$ defined by \eqref{C2} converges and $\sum_n\psi_{n,i}(\omega)\tau_n=\infty$ for each $i$;
$Y_{n+1,i}(\omega)-x^*_i$ defined by  \eqref{C1} converges  to zero, as in proof of Theorem \ref{Main-SA}. 
We will drop $\omega$ in the notation as usual, and we should keep in mind that choices of $n,\,m,\,K$ in what follows can depend upon $\omega$.

Let $m^*$ be chosen such that $\tau_n=b\beta_n<1$ for all $n\ge m^*$. Since $\{\beta_n:n\ge 0\}$ satisfies  \eqref{U8}, so does  $\{\tau_n:n\ge 0\}$. 
We can check, using recursion and \eqref{C1}
that for $ 0\le m\le n$,  we have
 \begin{equation}\label{C8}
 \begin{split}
Y_{n+1,i}-x^*_i=& \bigl[\prod_{j=m}^{n}(1-\psi_{j,i}\tau_j)(Y_{m,i}-x^*_i)\\&\;\;\;\;+ \sum_{k=m}^{n} \bigl[\prod_{_{\{j:k<j\le n\}}}(1-\psi_{j,i}\tau_j)\bigr]\psi_{k,i}\tau_kT_{k,i}\\
&\;\;\;\;+ \sum_{k=m}^{n} \bigl[\prod_{_{\{j:k<j\le n\}}}(1-\psi_{j,i}\tau_j)\bigr]\psi_{n,i}\phi_k\tau_k Z_{k+1}. 
\end{split}\end{equation}
Here, product over an empty set is taken as 1.
 Let us introduce notation for $0\le m\le n$,
 \begin{align}
  U_{m,n,i}&= \bigl[\prod_{j=m}^{n}(1-\psi_{j,i}\tau_j)(Y_{m,i}-x^*_i)\\
  V_{m,n,i}&= \sum_{k=m}^{n} \bigl[\prod_{_{\{j:k<j\le n\}}}(1-\psi_{j,i}\tau_j)\bigr]\psi_{k,i}\tau_k(Y_{m,i}-x^*_i-T_{k,i})\\
 L_{m,n,i}&= {\sum_{k=m}^n}\bigl[\prod_{\{k<j\le n\}}(1-\psi_{j,i}\tau_j)\bigr]\psi_{k,i}\phi_k\tau_k Z_{k+1,i}\\
  D_{m,n,i}&= {\sum_{k=m}^n}\bigl[\prod_{\{k<j\le n\}}(1-\psi_{j,i}\tau_j)\bigr]\psi_{k,i}\tau_k Z_{k+1,i}
\end{align}
Let us note (by recursion on $n$, as in \eqref{B2a}) that
 \begin{equation}\label{C9}
 S_{n+1,i}=  \bigl[\prod_{\{m\le j\le n\}}(1-\psi_{j,i}\tau_j)\bigr]S_{m,i}
+D_{m,n,i}.\ee
In view of \eqref{C8}, we have for any $0\le m\le n$
\[Y_{n+1,i}-x^*_i= U_{m,n,i}+ V_{m,n,i}+ L_{m,n,i}\]
and hence (noting that $m^*\le m\le n$ implies $(1-\psi_{j,i}\tau_j)>0$ for all $j$, $m\le j\le n$)
 \begin{equation*}
|Y_{n+1,i}-x^*_i|\le | U_{m,n,i}|+| V_{m,n,i}|+|L_{m,n,i}|,\;\;\;\forall m^*\le m\le n.
\end{equation*}
\begin{equation}\label{R43}
\nmi{Y_{n+1}-x^*}\le \nmi{U_{m,n}}+\nmi{V_{m,n}}+\nmi{L_{m,n}},\;\;\;\forall m^*\le m\le n.
\end{equation}
Let $\gamma_n=\max\{\nmi{Y_k-x^*}:\;0\le k\le n\}$ and $\gamma^*=\sup_n \gamma_n$. Note that for all $i$ and $k\le n$,
$|Y_{k,i}-x^*_i|\le\gamma_n$. By definition, we have $\nmi{Y_n}\le (\gamma_n+\nmi{x^*})$ and recalling that $\phi_n=(1+\max_{0\le k\le n}\nmi{Y_k})$,  it follows that
\begin{equation}\label{R45}
\phi_n\le (1+\gamma_n+\nmi{x^*})\;\;\;\;\forall n.\end{equation}
We will first prove that for all
$\epsilon>0$, $\exists$ $n_\epsilon $ such that 
\begin{equation}\label{R46}
| L_{m,n,i}|\le(1+\gamma_n+\nmi{x^*})\epsilon\;\;\;\;\;\forall n\ge m\ge n_\epsilon,\;1\le i\le d.
\end{equation}
Recall, we are working with $\omega$ such that $S_{n+1}=S_{n+1}(\omega)$ converges to zero and hence
given $\epsilon>0$,   we can choose $n_\epsilon>m^*$ such that 
\[
\nmi{S_{n+1}}\le  \frac{1}{2}\epsilon\;\;\;\forall n\ge n_\epsilon
\]
and thus $|S_{n+1,i}|\le  \frac{1}{2}\epsilon\;\;\;\forall n\ge n_\epsilon$.
In view of  \eqref{C9}, it follows that 
\begin{equation}\label{R41}
|D_{m,n+1,i}|\le  \epsilon\;\;\;\forall n\ge m\ge n_\epsilon.
\end{equation} 
As seen in proof of Theorem \ref{Main-SA}, it follows that 
\be\label{R11x}
\begin{split}
 L_{m,n,i}=& {\sum_{k=m}^n}\phi_k[\prod_{\{k<j\le n\}}(1-\psi_{j,i}\tau_j)]\psi_{k,i}\tau_k Z_{k+1,i}\\
=&\phi_m {\sum_{k=m}^n}[\prod_{\{k<j\le n\}}(1-\psi_{j,i}\tau_j)]\psi_{k,i}\tau_k Z_{k+1,i}\\
&\;\;\;\;\;\;\;\;\;\;\;\;\;\;\;\;+ {\sum_{k=m+1}^n}(\phi_k-\phi_m)[\prod_{\{k<j\le n\}}(1-\psi_{j,i}\tau_j)]\psi_{k,i}\tau_k Z_{k+1,i}\\
=&\phi_mD_{m,n}+ {\sum_{k=m+1}^n}(\phi_k-\phi_m)[\prod_{\{k<j\le n\}}(1-\psi_{j,i}\tau_j)]\psi_{k,i}\tau_k Z_{k+1,i}\\
\end{split}\ee
Proceeding likewise, we will get
\[ L_{m,n,i}=\phi_mD_{m,n}+(\phi_{m+1}-\phi_m)D_{m+1,n,i}+\ldots +(\phi_{n}-\phi_{n-1})D_{n,n,i}\]
Since $\phi_k$ is increasing, in view of \eqref{R41} it follows that $\forall n\ge m\ge n_\epsilon$, 
\[\begin{split} |L_{m,n,i}|&\le \phi_m|D_{m,n,i}|+(\phi_{m+1}-\phi_m)|D_{m+1,n,i}|+\ldots +(\phi_{n}-\phi_{n-1})|D_{n,n,i}|\\
&\le \phi_m  \epsilon+(\phi_{m+1}-\phi_m)  \epsilon+\ldots +(\phi_{n}-\phi_{n-1})  \epsilon\\
&=\phi_{n}  \epsilon\\
&\le(1+\gamma_n+\nmi{x^*})\epsilon\end{split}
\]
where we have used \eqref{R45} at the last step. 
This proves \eqref{R46}. Next we will prove that
\begin{equation}\label{R51}
\gamma^*=\sup\{\nmi{Y_n-x^*}:n\ge 0\}<\infty.
\end{equation}
 First, using $\rho<1$, get $\epsilon>0$   such that  $\epsilon<\frac{1}{3}$
\begin{equation}\label{R53}
\rho(1+3\epsilon)+\epsilon\le 1.
\end{equation}
For this $\epsilon$, choose $n_\epsilon$ such that \eqref{R46} holds and also
\begin{equation}\label{R47}
ac_n\le \epsilon \;\;\forall n\ge n_\epsilon,
\end{equation}
with $a$ such that \eqref{V1} holds.
This can be done since $c_n$ converges to zero in view of \eqref{V3}.

Now we will prove, by induction on $k$ that $\forall k\ge n_\epsilon$
\begin{equation}\label{R54}
\gamma_k\le (1+3\epsilon)(1+\gamma_{n_\epsilon}+\nmi{x^*}).
\end{equation}
Clearly \eqref{R54} is true for $k=n_\epsilon$.   
Suppose that \eqref{R54} holds for all $k$ such that  $n_\epsilon\le k\le n^*$.
We will prove that \eqref{R54} also holds for $k=n^*+1$, completing the proof that \eqref{R54} is true for all $k\ge n_\epsilon$.  To prove the induction step, suffices to prove  that for all $1\le i\le d$
\begin{equation}\label{R58}
|Y_{n^*+1,i}-x^*_i|\le (1+3\epsilon)(1+\gamma_{n_\epsilon}+\nmi{x^*_i}).\end{equation}
This follows trivially if $\psi_{n^*,i}=0$. When  $\psi_{n^*,i}=1$,
for $n_\epsilon\le k\le n^*$, as noted in \eqref{f2}
\begin{equation}\label{R59} \begin{split}
|Y_{k,i}-x^*_i-T_{k,i}|&\le \rho\nmi{Y_{k}-x^*}+ac_n\\
&\le \rho\gamma_k+ac_n\\
&\le \rho (1+3\epsilon)(1+\gamma_{n_\epsilon}+\nmi{x^*})+\epsilon \\
&\le (1+\gamma_{n_\epsilon}+\nmi{x^*})
\end{split}
\end{equation}
where we have used \eqref{R51} and \eqref{R47}. As a consequence
\begin{equation}\label{R61}\begin{split}
|V_{n_\epsilon,n^*,i}|&\le \textstyle\sum_{k=n_\epsilon}^{n^*}\textstyle\bigl[\prod_{_{\{j:k<j\le n\}}}(1-\psi_{j,i}\tau_j)\bigr]\psi_{k,i}\tau_k|Y_{k,i}-x^*_i-T_{k,i}|\\
&\le (1+\gamma_{n_\epsilon}+\nmi{x^*})\textstyle\sum_{k=n_\epsilon}^{n^*}\textstyle\bigl[\prod_{_{\{j:k<j\le n^*\}}}(1-\psi_{j,i}\tau_j)\bigr]\psi_{k,i}\tau_k.
\end{split}\end{equation}
It is easy to see that
\begin{equation}\label{R62}\begin{split}
|U_{n_\epsilon,n^*,i}|&=  \bigl[\prod_{j=n_\epsilon}^{n^*}(1-\psi_{j,i}\tau_j)\bigr]|Y_{n_\epsilon,i}-x^*_i|\\
   &\le \gamma_{n_\epsilon} \bigl[\prod_{j=n_\epsilon}^{n^*}(1-\psi_{j,i}\tau_j)\bigr].\end{split}
 \end{equation}
For all $m\le n$ and $1\le i\le d$ one can verify (by induction on $n$) that
\begin{equation}\label{R64}
 \Bigl(\bigl[\prod_{k=m}^{n}(1-\psi_{k,i}\tau_k)\bigr]+ \sum_{k=m}^{n} \bigl[\prod_{_{\{j:k<j\le n\}}}(1-\psi_{j,i}\tau_j)\bigr]\psi_{k,i}\tau_k\Bigr)=1,    
\end{equation}
and then putting together \eqref{R61}, \eqref{R62}, \eqref{R64} we conclude that
\begin{equation}\label{R67}
|U_{n_\epsilon,n^*,i}|+|V_{n_\epsilon,n^*,i}|\le (1+\gamma_{n_\epsilon}+\nmi{x^*}).\end{equation}
Using \eqref{R46}, along with \eqref{R54} for $k=n^*$ we deduce
  \begin{equation}\label{R68}
\begin{split}
|L_{n_\epsilon,n^*,i}|
&\le \epsilon(1+\gamma_{n^*}+\nmi{x^*})\\
&\le \epsilon(1+\nmi{x^*})+ \epsilon\gamma_{n^*}\\
&\le \epsilon(1+\nmi{x^*})+ \epsilon (1+3\epsilon)(1+\gamma_{n_\epsilon}+\nmi{x^*})\\
&\le (2\epsilon+3\epsilon^2)(1+\nmi{x^*})+\epsilon (1+3\epsilon)\gamma_{n_\epsilon}\\
&\le 3\epsilon(1+\gamma_{n_\epsilon}+\nmi{x^*}). 
\end{split}
\end{equation}
Here, since $\epsilon<\frac{1}{3}$, we have used $(2\epsilon+3\epsilon^2)\le 3\epsilon$ and also $(1+3\epsilon)\le 2$.
We thus conclude, using \eqref{R43}, \eqref{R67} along with \eqref{R68}, that
\begin{equation}\label{B71}\begin{split}
|Y_{n^*+1,i}-x^*_i|&\le | U_{n_\epsilon,n^*}|+| V_{n_\epsilon,n^*}|+|L_{n_\epsilon,n^*}|\\
&\le (1+3\epsilon)(1+\gamma_{n_\epsilon}+\nmi{x^*}).
\end{split}\end{equation}
Thus 
\[\nmi{Y_{n^*+1}-x^*}\le (1+3\epsilon)(1+\gamma_{n_\epsilon}+\nmi{x^*}).\]
This completes the induction proof and thus \eqref{R54} holds for all $k\ge n_{\epsilon}$  proving \eqref{R51}.

It also follows that $\sup\{\nmi{Y_k}: k\ge 0\}<\infty $ and hence that 
\begin{equation}\label{B72}
\sup\{\phi_{n}:n\ge 0\}<\infty.
\end{equation}
Using \eqref{R46} and \eqref{B72}, it follows that for all $\delta>0$, we can get $m_\delta$ such that
 \begin{equation}\label{B74}  
| L_{m,n,i}|<\delta \;\;\;\forall m,n \text{ such that }m_\delta\le m\le n, \;\forall 1\le i\le d
  \end{equation}
and thus we have for $m\ge m_\delta$, $1\le i\le d$
\[\limsup_{n \ap \infty}| L_{m,n,i}|\le\delta \;\;\;\;\;\forall i\]
and hence for all $m\ge m_\delta$
\begin{equation}\label{B75}
\limsup_{n \ap \infty}\nmi{L_{m,n,}}\le\delta.
\end{equation}
Since  $\sum_n\psi_{n,i}\tau_n=\infty$ we have
\begin{equation}\label{B77}
    \lim_{n \ap \infty}\bigl[\prod_{j=m}^{n}(1-\psi_{j,i}\tau_j)\bigr]=0, \;\;\;\forall m<\infty.
  \end{equation} 
For $m\ge 1$, let $\theta_m=\sup\{\nmi{Y_k-x^*}: k\ge m\}$. Then clearly, $\theta_m$ is decreasing and $\theta_m\le \gamma^*<\infty$ for all $m$. Let $\theta^*=\lim_{m \ap \infty}\theta_m$. Clearly $\theta^*\le \gamma^*$. Remains to show that $\theta^*=0$.
Since
\[ | U_{m,n,i}|\le |Y_{m,i}-x^*_i| \bigl[\prod_{j=m}^{n}(1-\psi_{j,i}\tau_j)\bigr]\]
and thus in view of \eqref{B77} we have
\begin{equation}\label{B78}
\lim_{n \ap \infty}\nmi{ U_{m,n}}=0 \;\;\;\;\;\forall m<\infty.
  \end{equation} 
Note that for any $m\le n$ we have
\begin{equation}\label{B79}\begin{split}
| V_{m,n,i}|&\le  \sum_{k=m}^{n} \bigl[\prod_{_{\{j:k<j\le n\}}}(1-\psi_{j,i}\tau_j)\bigr]\psi_{k,i}\tau_k\rho|Y_{k,i}-x^*_i|\\
&\le \sum_{k=m}^{n} \bigl[\prod_{_{\{j:k<j\le n\}}}(1-\psi_{j,i}\tau_j)\bigr]\tau_k\rho\theta_m \\
&\le \rho\theta_m \\
\end{split}\end{equation}   
Hence for all $m$, $\limsup_{n \ap \infty}| V_{m,n,i}|\le 
\rho\theta_m$ and as a consequence
\begin{equation}\label{B80}
\limsup_{n \ap \infty}\nmi{ V_{m,n}}\le 
\rho\theta_m.\end{equation}
Combining \eqref{R43} along with \eqref{B75}, \eqref{B78} and  \eqref{B80}, we get for any $m\ge m_\delta$ 
\[
\limsup_{n \ap \infty}\nmi{Y_{n+1}-x^*}\le \rho\theta_m +\delta.\]
Now taking limit as $m \ap \infty$ on the RHS above and then limit as $\delta\downarrow 0$, we get
\[\limsup_{n \ap \infty}\nmi{Y_{n+1}-x^*}\le \rho\theta^*.\]
From the definition of $\theta^*$, it follows that  $\limsup_{n \ap \infty}\nmi{Y_{n+1}-x^*}=\theta^*$ and hence we have shown
\begin{equation}\label{B90}  
\theta^*\le \rho\theta^*.
\end{equation} 
Since $\rho<1$ and $0\le\theta^*\le \gamma^*<\infty$, \eqref{B90} implies that $\theta^*=0$. This completes the proof of Theorem   \ref{Main-SGD}.
\end{proof}

Now we present a proof of Theorem \ref{GSLLN-details}.

\begin{proof}
\textbf{Of Theorem \ref{GSLLN-details}}.
We start by
noting that under conditions (H1) as well as the case $\alpha=2$ under (H5),
the conditions of Theorem \ref{GSLLN} are satisfied with $A_{n,i}^c=\es$
(the empty set).
Condition \eqref{G1} holds trivially and \eqref{G2a} holds because of
the condition on $\{W_n\}$ that $E(W_n)=0$ under  (H1) and $\{W_n\}$ being
a martingale difference sequence under $(H5)$.
Similarly, \eqref{G3} holds in view of \eqref{L1} under (H1) and \eqref{L5}
under (H5), in the case $\al = 2$.
Hence GSLLN holds. 
 
Let us now come to (H5) for any $\alpha\in (1,2)$. Let us define $A_{n,i}=\{\beta_n|W_{n,i}|\le 1\}$. Now $\P(A_{n,i}^c)=\P(\beta_n^\alpha|W_{n,i}|^\alpha>1)$ and thus $\P(A_{n,i}^c)\le \beta_n^\alpha E(|W_{n+1,i}|^\alpha)\le  \beta_n^\alpha\nu_{\alpha,n+1}$. Hence
  $\sum_n\P(A_{n,i}^c)<\infty$ in view of \eqref{L5} showing that \eqref{G1} holds in this case. Since $\{W_n:n\ge 1\}$ is a martingale difference sequence, we have
  \[ V_{n+1,i}= E[W_{n+1,i}\Ind_{A_{n+1,i}}\mid \F_n]= -   E[W_{n+1,i}\Ind_{A_{n+1,i}^c}\mid \F_n]\]
and hence
  \begin{align*}
\beta_n  E\bigl[|V_{n+1,i} |\bigr]&=\beta_n  E\bigl[| E[W_{n+1,i}\Ind_{A_{n+1,i}^c}\mid \F_n]|\bigr] \\
  &\le \beta_n  E\bigl[ E[|W_{n+1,i}|\Ind_{A_{n+1,i}^c}\mid \F_n]\bigr] \\
  &\le   E\bigl[ \beta_n|W_{n+1,i}|\Ind_{A_{n+1,i}^c}\bigr]\\
  &\le   E\bigl[ \beta_n^\alpha|W_{n+1,i}|^\alpha\Ind_{A_{n+1,i}^c}\bigr]
 \end{align*}
 since $a^\alpha>a$ when $a>1$ and $1<\alpha$. Thus
 \[ E[ \sum_n\beta_n|V_{n+1,i} |]\le  E[\sum_n\beta_n^\alpha|W_{n+1,i}|^\alpha]<\infty\]
 in view of the assumption \eqref{L5}. Thus \eqref{G2} holds. Coming to the third condition, using
  $a^\alpha\ge a^2$ when $0<a\le 1$ and $1<\alpha< 2$, we have
 \begin{align*}
  \sum_n\beta_n^2| E(|V_{n+1,i}|^2 )|\le &  \sum_n\beta_n^\alpha E(|W_{n+1,i}|^\alpha\Ind_{A_{n+1,i}})\\
 \le & \sum_n\beta_n^\alpha E(|W_{n+1,i}|^\alpha) .
 \end{align*}
Thus \eqref{L5} implies \eqref{G3}.
Thus GSLLN holds if (H5) holds. Coming to (H4), it is easy to see that
$U_n=W_n-E(W_n)$ is a martingale difference sequence and that $\{U_n:n\ge 0\}$
satisfies (H5) (with the same $\alpha$.
Thus $\{U_n:n\ge 0\}$ satisfies GSLLN for the chosen step size
$\{\beta_n:n\ge 0\}$ since $\mu_n$ converges to 0, see part (i)  of
Lemma \ref{deterministic}.  Thus GSLLN holds under (H4).
 
Let us now look at (H2). This includes the case $\alpha=1$ while in (H5), we need $1<\alpha\le 2$. Moreover,  $\beta_n =n^{-\frac{1}{\alpha}}$  satisfies (L2) but not (L5).  
 
  Let us define $A_{n+1,i}=\{|W_{n+1,i}|^\alpha\le n\}$. Since $\{W_n:n\ge 1\}$ are i.i.d. we have
  \[\P(A_{n+1,i}^c)=\P(|W_{n+1,i}|^\alpha> n)= \P(|W_{1,i}|^\alpha> n).\]
Since $ E(|W_{1,i}|^\alpha)<\infty$, it follows that $\sum_n \P(|W_{1,i}|^\alpha> n)<\infty$ and hence \eqref{G1} holds. 

Since $\{W_n:n\ge 1\}$ are i.i.d. and $ E[W_{1,i}]=0$, we have
\[V_{n+1,i}= E[W_{n+1,i}\Ind_{A_{n+1,i}}\mid \F_n]= E[W_{n+1,i}\Ind_{A_{n+1,i}}]= E[W_{1,i}\Ind_{\{ |W_{1,i}|^\alpha\le n\}}]=- E[W_{1,i}\Ind_{\{ |W_{1,i}|^\alpha> n\}}].\]
Thus 
\[|V_{n+1,i}|\le  E[|W_{1,i}|\Ind_{\{ |W_{1,i}|^\alpha> n\}}]\le  E[|W_{1,i}|^\alpha\Ind_{\{ |W_{1,i}|^\alpha> n\}}].\]
Since $ E[|W_{1,i}|^\alpha]<\infty$ and  we conclude $|V_{n+1,i}|$ converges to zero and hence \eqref{G2a} holds. 

Coming to the \eqref{G3}, we have (recall that  $1\le \alpha<2$)
 \begin{align*}
 \sum_{n=0}^\infty \beta_n^2 E[|W_{n+1,i}|^2\Ind_{A_{n+1,i}}]&\le C_3+  \sum_{n=2}^\infty D^2n^{-\frac{2}{\alpha}} E[|W_{1,i}|^2\Ind_{\{|W_{1,i}|^\alpha\le  n\}}]\\
&=C_3+   D^2 E[(|W_{1,i}|^2\sum_{n=2}^\infty n^{-\frac{2}{\alpha}}\Ind_{\{|W_{1,i}|^\alpha\le  n\}}]\\
\end{align*}
\vskip -11mm 
Noting that for $\theta>1$ and $T>0$ we have
\bd
\sum_{n=2}^\infty n^{-\theta}\Ind_{\{T\le  n\}}\le 2T^{1-\theta} ,
\ed
we can conclude that
 \begin{align*}
 \sum_{n=0}^\infty \beta_n^2 E[|W_{n+1,i}|^2\Ind_{A_{n+1,i}}]
&\le C_3+ 2 D^2 E[(|W_{1,i}|^2(|W_{1,i}|^\alpha)^{(1-\frac{2}{\alpha})}\\
&=C_3+ 2 D^2 E[(|W_{1,i}|^\alpha)]\\
&<\infty
 \end{align*}
 in view of the assumption that  $ E(\nmm{W_1}^\alpha)<\infty$
 and hence \eqref{G3} follows.
Thus the conclusion that GSLLN holds under (H2) follows from Theorem \ref{GSLLN2}.

It remains only to show that conclusion is valid if (H3) holds.
For this purpose, we define,
\bd
A_{n,i}= \left\{\frac{|W_{n,i}|}{({\log(1+|W_{n,i}|)})^\delta}\le n
\right\} , \fa n \ge 1 .
\ed
Since  
\bd
E[\frac{|W_{n,i}|}{({\log(1+|W_{n,i}|)})^\delta}]<\infty \fa i ,
\ed
it follows that $\sum_i\P( A_{n+1,i})<\infty$ and hence \eqref{G1} holds.
Since   $\{W_n: n\ge 1\}$ are i.i.d.\ with symmetric distribution
(that is, $W_1$ and $-W_1$ have the same distribtuion),
it follows that $W_{n,i}\Ind_{A_{n+1,i}}$ has symmetric distribution for each
$i$; hence \eqref{G2} holds trivially.
Coming to \eqref{G3}, it remains to show that
\bd
\sum_{n=0}^\infty \beta_n^2 E[|W_{n+1,i}|^2\Ind_{A_{n+1,i}}]<\infty
\ed
In view of \eqref{L3}, it suffices to prove that
\be\label{Rx1}  
\sum_{n=3}^\infty \frac{1}{n^2(\log(1+n))^{2\delta}} E[|W_{1,i}|^2\Ind_{\{\frac{|W_{1,i}|}{(\log(1+|W_{1,i}|))^\delta}\le n\}}]<\infty.
\ee
As seen in Lemma \ref{Log},
\bd
\log(1+|W_{1,i}|)\le 3 \log(1+n) \mbox{ on the set }
\left\{\frac{|W_{1,i}|}{(\log(1+|W_{1,i}|))^\delta}\le n \right\} .
\ed
Thus
\bd
\frac{1}{(\log(1+|W_{1,i}|))^{2\delta}}\ge \frac{1}{9(\log(1+n))^{2\delta}}.
\ed
Hence to prove \eqref{Rx1}, it suffices to prove that
\be\label{Rx2}  
 \sum_{n=3}^\infty \frac{1}{n^2} E[\frac{|W_{1,i}|^2}{({\log(1+|W_{1,i}|)})^{2\delta}}
\Ind_{\left\{\frac{|W_{1,i}|}{(\log(1+|W_{1,i}|))^\delta}\le n\right\}}]<\infty.
\ee
Since $|W_{1,i}|/(\log(1+|W_{1,i}|))^\delta$ is integrable, \eqref{Rx2} follows as seen above in the proof under condition (H2) for $\alpha=1$. 
\end{proof}

\begin{remark}\label{remark}
It should be noted that under (H1), (H4) and (H5) we have invoked Theorem
\ref{GSLLN} and hence we can assert that
\be\label{G7x}
\sum_{n=0}^\infty \beta_nW_{n+1,i}\;\;\text{converges}\;\;\;\forall i.
\ee
holds.
Under (H2) if the (common) distribution of $\{W_n:n\ge 1\}$  is symmetric,
then  it is clear that \eqref{G2} trivially holds and again we can invoke
Theorem \ref{GSLLN}  instead of Theorem \ref{GSLLN2} and again assert that
\eqref{G7x} holds. 
Under (H3) in any case, the common distribution is assumed to be symmetric and hence \eqref{G7x} holds.
\end{remark}

Next, we present a proof of  Theorem \ref{GSLLN-SGD}.
\begin{proof}
\textbf{Of Theorem \ref{GSLLN-SGD}.}
Let us write $\theta_n = \eta_n/c_n$.
Recall that $M_n=\frac{1}{2}(\Mp_{n} - \Mpp_{n})$.
It is easy to see that if  $\{\Mp_{n},\Mpp_{n}:n\ge 1\}$ satisfy (K1),
then $\{M_n: n\ge 1\}$ and $\{\theta_n:n\ge 0\}$ satisfy (H1).
The same holds for (K2) or (K3), in which case $\{M_n: n\ge 1\}$
and $\{\theta_n:n\ge 0\}$ satisfy (H2) or (H3) respectively. 
Similarly, if (K4) or (K5) hold, then again  $\{M_n: n\ge 1\}$ and
$\{\theta_n:n\ge 0\}$ satisfy (H4) or (H5) respectively. 
Since $M_n$ by definition here have a symmetric distribution under
(H1), (H2) and (H3), Theorem \ref{GSLLN-details} yields that
(see also Remark \ref{remark} above)
\bd
\sum_{n=0}^\infty M_{n+1}\theta_n \text{ converges almost surely}.
\ed
 Noting that $W_{n+1}\eta_n=M_{n+1}\theta_n$,  we conclude that
\bd
\sum_{n=0}^\infty W_{n+1}\eta_n \text{ converges almost surely} .
\ed
and hence $\{W_n:n\ge 1\}$ satisfies GSLLN at the rate $\{\eta_n:n\ge 1\}$.
As a consequence, if in addition, \eqref{V1}, \eqref{V2}, \eqref{V3},  
and \eqref{V3x} are satisfied, then the
Stochastic Gradient Descent Theorem \ref{Main-SGD} holds.
\end{proof}


\section{Conclusions and Future Research}\label{sec:Conc}

In this paper, we have introduced a new approach to proving the convergence
of the Stochastic Approximation (SA)
and the Stochastic Gradient Descent (SGD) algorithms.
The new approach is based on a concept called GSLLN (Generalized Strong
Law of Large Numbers), which extends the traditional SLLN.
Using this concept, we provide sufficient conditions for convergence,
which effectively decouple the properties of the function whose zero
we are trying to find, from the properties of the measurement errors (noise
sequence).
The new approach provides an alternative to the two widely used approaches,
namely the ODE approach and the martingale approach, and also permits
a wider class of noise signals than either of the two known approaches.
In particular, the ``noise'' or measurement error \textit{need not}
have a finite second moment, and under suitable conditions, not even
a finite mean.
By adapting this method of proof, we also derive sufficient conditions
for the convergence of zero-order SGD, wherein the stochastic gradient
is computed using $2d$ function evaluations, but
no gradient computations.
The sufficient conditions derived here are the weakest to date,
thus leading to a considerable expansion of the applicability of SA and SGD
theory.

While the result on the convergence of the SGD algorithm is novel
in terms of its breadth, it is not practically useful unless the dimension
of the problem $d$ is quite small.
This is because it requires $2d$ function evaluations.
The state of the art in SGD is an approach known as SPSA (Simultaneous
Perturbation Stochactic Approximation), which requires
\textit{only two} function evaluations, irrespective of the value of $d$.
This approach was introduced in \cite{Spall-TAC92}, with several 
improvements thereafter.
See \cite{Sadegh-Spall-TAC98} for an improved version of SPSA,
and \cite{Li-Xia-Xu-arxiv22} for a review of the literature on this topic.
The authors are currently investigating how the GSLLN approach can
be adjusted to reduce the number of function measurements to just two,
for every value of $d$.


\begin{thebibliography}{10}

\bibitem{BMP90}
Albert Benveniste, Michel M\'{e}tivier, and Pierre Priouret.
\newblock {\em Adaptive Algorithms and Stochastic Approximation}.
\newblock Springer-Verlag, 1990.

\bibitem{Bertsekas19}
D.~P. Bertsekas.
\newblock {\em Reinforcement Learning and Optimal Control}.
\newblock Athena Scientific, 2019.

\bibitem{Blum54b}
Julius~R. Blum.
\newblock Approximation methods which converge with probability one.
\newblock {\em Annals of Mathematical Statistics}, 25:382--386, 1954.

\bibitem{Blum54}
Julius~R. Blum.
\newblock Multivariable stochastic approximation methods.
\newblock {\em Annals of Mathematical Statistics}, 25(4):737--744, 1954.

\bibitem{Borkar98}
V.~S. Borkar.
\newblock Asynchronous stochastic approximations.
\newblock {\em {SIAM Journal on Control and Optimization}}, 36(3):840--851,
  1998.

\bibitem{Borkar-Meyn00}
V.~S. Borkar and S.~P. Meyn.
\newblock {The O.D.E. method for convergence of stochastic approximation and
  reinforcement learning}.
\newblock {\em {SIAM Journal on Control and Optimization}}, 38:447--469, 2000.

\bibitem{Borkar08}
Vivek~S. Borkar.
\newblock {\em Stochastic Approximation: A Dynamical Systems Viewpoint}.
\newblock Cambridge University Press, 2008.

\bibitem{Borkar22}
Vivek~S. Borkar.
\newblock {\em Stochastic Approximation: A Dynamical Systems Viewpoint (Second
  Edition)}.
\newblock Cambridge University Press, 2022.

\bibitem{Breiman92}
Leo Breiman.
\newblock {\em Probability}.
\newblock {SIAM: Society for Industrial and Applied Mathematics}, 1992.

\bibitem{Chung51}
K.~L. Chung.
\newblock The strong law of large numbers.
\newblock In {\em Proceedings of the Second Berkeley Symposium on Mathematical
  Statistics and Probability}, pages 341--352, 1951.

\bibitem{Dvoretzky56}
A.~Dvoretzky.
\newblock On stochastic approximation.
\newblock In {\em Proceedings of the Third Berkeley Symposium on Mathematical
  Statististics and Probability}, volume~1, pages 39--56. University of
  California Press, 1956.

\bibitem{Gladyshev65}
E.~G. Gladyshev.
\newblock On stochastic approximation.
\newblock {\em Theory of Probability and Its Applications}, X(2):275--278,
  1965.

\bibitem{Goswami-Rao25}
A.~Goswami and B.~V. Rao.
\newblock {\em Measure Theory for Analysis and Probability}.
\newblock Springer, 2025.

\bibitem{Jaakkola-et-al94}
Tommi Jaakkola, Michael~I. Jordan, and Satinder~P. Singh.
\newblock Convergence of stochastic iterative dynamic programming algorithms.
\newblock {\em Neural Computation}, 6(6):1185--1201, November 1994.

\bibitem{RLK-BVR24}
R.~L. Karandikar and B.~V. Rao.
\newblock Stochastic approximation in infinite dimensions.
\newblock {\em Infinite Dimensional Analysis, Quantum Probability and Related
  Topics}, 27, 2024.

\bibitem{MV-RLK-SGD-JOTA24}
Rajeeva~L. Karandikar and M.~Vidyasagar.
\newblock Convergence rates for stochastic approximation: Biased noise with
  unbounded variance, and applications.
\newblock {\em Journal of Optimization Theory and Applications},
  203:2412--2450, October 2024.

\bibitem{MV-RLK-BASA-arxiv21}
Rajeeva~L. Karandikar and M.~Vidyasagar.
\newblock Recent advances in stochastic approximation with applications to
  nonconvex optimization and fixed point problems.
\newblock https://arxiv.org/pdf/2109.03445v6.pdf, February 2024.

\bibitem{Kief-Wolf-AOMS52}
J.~Kiefer and J.~Wolfowitz.
\newblock Stochastic estimation of the maximum of a regression function.
\newblock {\em Annals of Mathematical Statistics}, 23(3):462--466, 1952.

\bibitem{Kushner-Shwartz85}
H.~J. Kushner and A.~Shwartz.
\newblock {Stochastic Approximation in Hilbert Space: Identification and
  Optimization of Linear Continuous Parameter Systems}.
\newblock {\em {SIAM Journal on Control and Optimization}}, 2(5):774--793,
  1985.

\bibitem{Kushner-JMAA77}
Harold~J. Kushner.
\newblock General convergence results for stochastic approximations via weak
  convergence theory.
\newblock {\em Journal of Mathematical Analysis and Applications},
  61(2):490--503, 1977.

\bibitem{Kushner-Clark78}
Harold~J. Kushner and Dean~S. Clark.
\newblock {\em Stochastic Approximation Methods for Constrained and
  Unconstrained Systems}.
\newblock Applied Mathematical Sciences. Springer-Verlag, 1978.

\bibitem{Kushner-Clark12}
Harold~J. Kushner and Dean~S. Clark.
\newblock {\em Stochastic approximation methods for constrained and
  unconstrained systems}.
\newblock Springer Science \& Business Media. Springer-Verlag, 2012.

\bibitem{Kushner-Yin97}
Harold~J. Kushner and G.~George Yin.
\newblock {\em Stochastic Approximation Algorithms and Applications}.
\newblock Springer-Verlag, 1997.

\bibitem{Kushner-Yin03}
Harold~J. Kushner and G.~George Yin.
\newblock {\em Stochastic Approximation Algorithms and Applications (Second
  Edition)}.
\newblock Springer-Verlag, 2003.

\bibitem{Lai03}
Tze~Leung Lai.
\newblock Stochastic approximation (invited paper).
\newblock {\em The Annals of Statistics}, 31(2):391--406, 2003.

\bibitem{Li-Xia-Xu-arxiv22}
Shiru Li, Yong Xia, and Zi~Xu.
\newblock Simultaneous perturbation stochastic approximation: towards
  one-measurement per iteration.
\newblock arxiv:2203.03075, March 2022.

\bibitem{Ljung-TAC77a}
Lennart Ljung.
\newblock On positive real transfer functions and the convergence of some
  recursive schemes.
\newblock {\em {IEEE Transactions on Automatic Control}}, 22(6):539--551, 1977.

\bibitem{Ljung78}
Lennart Ljung.
\newblock Strong convergence of a stochastic approximation algorithm.
\newblock {\em Annals of Statistics}, 6:680--696, 1978.

\bibitem{Meti-Priou84}
Michel M\'{e}tivier and Pierre Priouret.
\newblock Applications of kushner and clark lemma to general classes of
  stochastic algorithms.
\newblock {\em {IEEE} Transactions on Information Theory}, IT-30(2):140--151,
  March 1984.

\bibitem{Milz23}
J.~Milz.
\newblock Sample average approximations of strongly convex stochastic programs
  in hilbert spaces.
\newblock {\em Optimization Letters}, 17:471--492, 2023.

\bibitem{Robb-Sieg71}
H.~Robbins and D.~Siegmund.
\newblock {\em A convergence theorem for non negative almost supermartingales
  and some applications}, pages 233--257.
\newblock Elsevier, 1971.

\bibitem{Robbins-Monro51}
Herbert Robbins and Sutton Monro.
\newblock A stochastic approximation method.
\newblock {\em Annals of Mathematical Statistics}, 22(3):400--407, 1951.

\bibitem{Sadegh-Spall-TAC98}
P.~Sadegh and J.~C. Spall.
\newblock Optimal random perturbations for stochastic approximation using a
  simultaneous perturbation gradient approximation.
\newblock {\em {IEEE transactions on automatic control}}, 43(10):1480--1484,
  1998.

\bibitem{Schmett60}
L.~Schmetterer.
\newblock Stochastic approximation.
\newblock In {\em Proceedings of the Fourth Berkeley Symposium on Mathematical
  Statistics and Probability}, pages 587--609, 1960.

\bibitem{Spall-TAC92}
J.C. Spall.
\newblock {Multivariate stochastic approximation using a simultaneous
  perturbation gradient approximation}.
\newblock {\em IEEE Transactions on Automatic Control}, 37(3):332--341, 1992.

\bibitem{Tsi-Van-TAC97}
John~N. Tsitsiklis and Benjamin~Van Roy.
\newblock An analysis of temporal-difference learning with function
  approximation.
\newblock {\em {IEEE Transactions on Automatic Control}}, 42(5):674--690, May
  1997.

\end{thebibliography}

\end{document}